\documentclass[11pt]{amsart}
\usepackage{soul}
\usepackage[T1]{fontenc}
\usepackage[english]{babel}
\usepackage{amsbsy,amssymb,url}
\usepackage{enumitem}
\usepackage{graphicx}
\usepackage{xcolor}
\usepackage{amsmath}
\usepackage{mathrsfs}
\usepackage{mathtools}
\usepackage{amsfonts,amssymb,latexsym,MnSymbol,dsfont}%,graphicx,psfrag}
\usepackage{geometry}
\usepackage{xcolor}
\newtheorem{theorem}{Theorem}[section]

\newtheorem{corollary}[theorem]{Corollary}

\newtheorem{proposition}[theorem]{Proposition}
\newtheorem{definition}[theorem]{Definition}
\newtheorem{remark}[theorem]{Remark}

\numberwithin{equation}{section}
\newcommand{\R}{\mathbb{R}}

\newcommand{\B}{\mathbb{B}}
\newcommand*{\dd}{\mathop{}\!\mathrm{d}}
\newcommand{\Ell}{\mathscr{L}}

\begin{document}

\title[Rotational controls and viscosity solutions]{Rotational controls and uniqueness of constrained viscosity solutions of Hamilton-Jacobi PDE}
\author[Giovanni Colombo]{Giovanni Colombo}
\address[Giovanni Colombo]{Universit\`a\ di Padova, Dipartimento di Matematica ``Tullio Levi-Civita'', 
via Trieste 63, 35121 Padova, Italy, and Istituto Nazionale di Alta Matematica ``F. Severi'', Unit\`a\ di ricerca presso l'Universit\`a\
di Padova}
\email{colombo@math.unipd.it}
\author[Nathalie T. Khalil]{Nathalie T. Khalil}
\address[Nathalie T. Khalil]{Research Center for Systems and Technologies SYSTEC, Faculty of Engineering, University of Porto, Rua Roberto Frias s/n 4200-465 Porto, Portugal}
\email{khalil.t.nathalie@gmail.com}
\author[Franco Rampazzo]{Franco Rampazzo}
\address[Franco Rampazzo]{Universit\`a\ di Padova, Dipartimento di Matematica ``Tullio Levi-Civita'', 
via Trieste 63, 35121 Padova, Italy}
\email{rampazzo@math.unipd.it}

\date{\today}

\begin{abstract}
The classical  {\it inward pointing condition} (IPC) for a control system whose state $x$ is constrained in the closure
$C:=\bar\Omega$ of an open set $\Omega$ prescribes that at each point of the  boundary $x\in \partial \Omega$ the
intersection between the dynamics and the interior of the tangent space of $\bar \Omega$ at $x$ is nonempty. 
Under this hypothesis, for every system  trajectory $x(.)$ on  a time-interval $[0,T]$, possibly violating the constraint,  
one can  construct a  new system trajectory  $\hat x(.)$ that satisfies the constraint and whose distance from $x(.)$
is bounded by a quantity proportional to the  maximal deviation $d:=\mathrm{dist}(\Omega,x([0,T]))$.  
When (IPC)  is violated, the construction of such a constrained  trajectory is not possible in general. 
However, for a control system of the form $\dot{x}=f_1(x)u_1+f_2(x)u_2$, we prove in this paper
that a ``higher order'' inward pointing condition involving Lie brackets 
of the dynamics' vector fields allows for
a novel construction of  a constrained trajectory $\hat x(.)$ whose distance from  
the reference trajectory $x(.)$ is  bounded by a quantity proportional to $\sqrt{d}$. Our method requires 
a further assumption of non-positiveness of a sort of curvature and is based on the implementation 
of  a suitable ``rotating'' control strategy. 
As an application, we establish the continuity \textit{up to the boundary} of the value function $V$ of a 
classical state constrained optimal control problem, a property that allows to regard $V$ as the {\it unique}  
constrained viscosity solution of the corresponding Bellman equation.\\
{\sc R\'esum\'e.} La {\it condition classique de pointe intérieure} pour un système de contrôle ayant comme variable d'état $x$, contrainte à la fermeture $C:=\bar \Omega$ d'un ensemble ouvert $\Omega$, assume que pour chaque point appartenant à la frontière $x\in \partial \Omega$, l'intersection entre la dynamique et l'intérieur de l'espace tangent de $\bar \Omega$ au point $x$ n'est pas vide. Sous cette hypothèse, pour chaque trajectoire $x(.)$ sur l'intervalle de temps $[0,T]$, possiblement violant la contrainte d'état, une nouvelle trajectoire $\hat x(.)$ satisfaisant la contrainte d'état pourrait être construite, et pour cette trajectoire, la distance de $x(.)$ est bornée par une quantité proportionelle à la déviation maximale $d:= \textrm{dist}(\Omega, x([0,T]))$. Par contre, quand la condition classique de pointe intérieure n'est pas satisfaite, la construction d'une telle trajectoire n'est pas possible en général. Dans cet article, pour un système de contrôle de la forme $\dot x=f_1(x)u_1+f_2(x)u_2$, nous démontrons qu'une condition de pointe intérieure ``d'order supérieur'' qui tient en compte le crochet de Lie du champ de vecteurs de la dynamique permet une nouvelle construction d'une trajectoire $\hat x(.)$ appartenant à la contrainte d'état, et dont la distance de la trajectoire de référence $x(.)$ est bornée par une quantité proportionelle à $\sqrt{d}$. Notre méthode nécessite une hypothèse supplémentaire de courbature non positive, et est basée sur l'implémentation d'une stratégie de contrôle ``rotationel''. Comme application, on établit la continuité de la fonction valeur $V$, {\it pour des points appartenant à la frontière} de l'ensemble de contrainte, et ceci pour des problèmes de contrôle optimal classique. Sous cette continuité, la fonction valeur sera {\it l'unique} solution de viscosité de l'équation de Bellman correspondante.
\end{abstract}
\maketitle

\section{Introduction}
Several advances in the area of controlled dynamical systems
 \begin{equation}\label{gen_dyn}
 	\dot{x}(t)= f(x(t),u(t) ), \ u\in U,\quad  t\in [0,T]
 \end{equation}
subject to a  smooth state constraint
 \begin{equation}\label{st_constr}
 	x(t)\in C:=\{ z : h(z)\le 0\}\ \forall t\in [0,T]
 \end{equation}  ($h$ a smooth function such that $\nabla h(x)\neq 0$ $\forall x\in \partial C$) include the existence of
solutions, necessary optimality conditions for the minimizer of some objective function (that may be degenerate), and
the regularity of the corresponding value function.
For many of these issues, a crucial hypothesis is the so called {\it inward pointing condition}  (IPC), 
that assumes the existence of a constant $\gamma>0$ such that, at each point $x$ of 
the constraint's boundary, one can select a control $\bar u$ verifying
\begin{equation}\label{classical_IPC}
\nabla h(x) \cdot f(x,\bar u) \le - \gamma.
\end{equation}
This condition was introduced by Petrov in \cite{petr} in connection with the Lipschitz continuity of the minimum
time function and used by Soner in \cite{Soner_1986,Soner_1986_II}, in connection with the regularity of the value function of
a state constrained optimal control problem. In turn, such continuity is essential for the value function to be the unique 
{\it constrained viscosity solution} of the corresponding Hamilton-Jacobi-Belmann equation. 
An analogous result is obtained in \cite{Frankowska_Vinter_2000} in terms 
of a special class of lower semicontinuous solutions to the HJB equation.

Let us mention that the main fact guaranteed by assumption \eqref{classical_IPC} consists in the possibility, given a trajectory
$x$ of \eqref{gen_dyn} that violates the state constraint \eqref{st_constr}, of constructing another trajectory 
$\hat{x}$ of \eqref{gen_dyn} such that $\hat{x}(t)\in C$ for all $t\in [0,T]$ and, moreover, 
\begin{equation}\label{linW11}
\| \hat{x} - x\|_{W^{1,1}}\le K \max_{t\in [0,T]} \{\max \{h(x(t)),0\} \}
\end{equation}
for a suitable constant $K$ independent of $x$. A trajectory with the properties that are satisfied by $\hat{x}$ is usually called a
\textit {neighboring feasible trajectory} relative to $x$. 
Observe that the estimate \eqref{linW11} is \textit{linear} with respect to the
maximal deviation $d:= \max_{t\in [0,T]} \max \{h(x(t)),0\}$ of $x(.)$ from the constraint.
With the help of neighboring feasible trajectories, many applications to state constrained optimal control problems
were obtained in the literature. Besides the mentioned application to HJB equation, applications are given in 
\cite{Rampazzo_Vinter_2000, Lopes_Fontes_dePinho_2011, Arutyunov_Karamzin_2020_nondegeneracy} in deriving 
refined and unrestrictive conditions under which first order
necessary conditions (the Pontryagin Maximum Principle or the Generalized Euler-Lagrange condition) are
nondegenerate 
(i.e., the maximum principle is informative) or normal \cite{Bettiol_Khalil_Vinter_2016} (i.e., the Lagrange multiplier
associated to the cost function is not zero), characterizing value functions as the unique solution to the Hamilton-Jacobi
equation \cite{Soner_1986, Frankowska_Vinter_2000}, and establishing ``sensitivity relations'' in which the costate
trajectory and the Hamiltonian are interpreted in terms of generalized gradients of the value function 
\cite{Bettiol-Vinter2010-sensitivity analysis, Frankowska_2010_survey}. 
Soner's (IPC) was later adapted to differential inclusions with nonsmooth constraints and to problems 
depending measurably on the time variable, and sharp counterexamples were found. 
We refer the reader to
\cite{Rampazzo_Vinter_2000, Bettiol_Bressan_Vinter_2010, Bressan_Facci_2011, Bettiol_Khalil_Vinter_2016} 
and the references therein for an account of the literature.

In all the above results, given a reference process $(x,u)$, the control $\bar u$ that satisfies \eqref{classical_IPC} is
used to construct a feasible path $\hat x(.)$ that, for a time short enough, enters the interior
of the constraint. 
Then $\hat x$ is continued by using the reference control $u(.)$ and the trajectory constructed in this way is
proved to be feasible in the whole interval $[0,T]$, by possibly repeating finitely many times the above construction. 

When \eqref{classical_IPC} fails, simple examples show that it is not possible, in general, to construct
neighboring feasible trajectories, and moreover several of the above results are no longer valid.
In this paper we address precisely this situation, and construct neighboring feasible trajectories for a 
class of symmetric and control affine dynamical systems, where the classical inward pointing condition (IPC) is substituted by a
substantially weaker one. More precisely, we consider the system
\begin{equation}\label{dynamics}
\dot x = f_1(x)u_1+f_2(x)u_2, \ (u_1,u_2)\in U\subset \mathbb{R}^2,\;\text{ with }0\in\mathrm{int}\, U
\end{equation}
still with a smooth state constraint $C:=\{x: h(x)\le 0\}$. We call {\it singular set} the subset $\mathbb{S}$ of the state
space comprising the points $x$ where \eqref{classical_IPC} is violated,
namely, $\mathbb{S}:=\{x :\,\,\,\nabla h(x) \cdot f_1(x)=\nabla h(x) \cdot f_2(x)=0\}$. 
The condition which replaces $(IPC)$ simply 
states that for every $x\in\mathbb{S}$ the Lie bracket $[f_1,f_2](x)$ is not tangent
to the constraint, i.e.,
\begin{equation}\label{secord}
\nabla h(x)\cdot [f_1,f_2](x) \neq 0.
\end{equation}
To begin with, let us point out that both the classical (IPC)  and the ``second order'' condition  
\eqref{secord} are intrinsic, i.e., they are chart-invariant, so that the various issues in the quoted literature 
and in the present paper can be generalized to differential manifolds. Condition \eqref{secord}
was introduced by M. Motta in \cite{motta}, where it was proved that \eqref{secord} implies that
the value function of a discounted infinite horizon integral problem subject to the constraint $x\in C$ is continuous 
{\it  in the interior} of $C$. If the
starting point of the reference trajectory lies in the interior of the constraint,
in \cite{motta} one approximates the direction of $[f_1,f_2]$ by
standard  concatenations of the flows generated by $f_1$, $f_2$, $-f_1$, and $-f_2$. 
However, this method does not work as soon as the starting point lies on the boundary $\partial C$. 
Indeed, if both $f_1$ and $f_2$ are tangent
to the constraint's boundary, such an approach is doomed to fail, as a trajectory that follows the flow of either vector fields
for a while may immediately leave $C$.

Aiming at extending Motta's result to the boundary $\partial C$, we base our approach 
on a  \textit{continuous combination}  of $f_1$ and $f_2$ by using a control 
that rotates (a ``screwer'') with a sufficiently high frequency. The core of the proof consists in  
showing that the corresponding trajectory spirals towards the interior of $C$. 
Actually, for this to happen, in addition to the transversal bracket condition \eqref{secord}  we need to impose two more 
assumptions on the singular set $\mathbb{S}$.
The most crucial of these hypotheses  consists in the non-positiveness of a sort of curvature term expressed 
as a quadratic form $\mathbf{S}=\mathbf{S}(x): \mathbb{R}^2\to\mathbb{R}$, whose components are invariant scalars 
depending on the vector fields $f_1,f_2$, their differentials, and the first and second differential of the constraint
function $h$ (see assumption  (H5) in Section \ref{sec:ass}). 
The second condition implies that the trajectories of \eqref{dynamics} leave sufficiently quickly the singular 
set $\mathbb{S}$ (see (H6) in Section \ref{sec:ass}). 

Under the above hypotheses we are able to construct neighboring feasible trajectories that approximate a given reference path 
of the system. It must be said, however, that our neighboring property is weaker than \eqref{linW11}, in that
$\|\hat{x}-x\|_{W^{1,1}}$ is bounded above by a constant multiplied 
by the \textit{square root} of the maximal deviation $d$ of the reference trajectory from the constraint 
(while under the (IPC) this deviation grows as $d$). This is not surprising, as a time
of order $\sqrt{d}$ is needed by the spiraling trajectory in order to move sufficiently far towards the interior of $C$.
Nevertheless, this estimate is enough for some purposes. Indeed, as an application we prove a continuity result 
for the value function of a minimum problem, which in turn allows us  to characterize this function 
as the {\it unique} constrained viscosity solution of a suitable Hamilton-Jacobi-Bellman equation.

Finally, it is plausible that other applications to constrained control problems can be obtained by using our
construction of neighboring feasible trajectories, for instance for issues connected with the Maximum Principle (see e.g.\cite{Vinter_book_2010}), like normality or non-degeneracy 
of optimal trajectories of state-constrained optimal control problems. Moreover, it is expected that a generalization of our
construction to the case of transversality conditions involving Lie brackets of arbitrary order can be obtained.

The paper is organized as follows: the assumptions and the main results are collected in Section \ref{sec:ass}, 
while motivational examples and sharpness of some assumptions are discussed in Section \ref{sec:examples}. 
The construction of the neighboring feasible trajectory is performed in Section \ref{sec:proofNFT}.
In particular, the delicate estimates that are concerned with the spiraling trajectory are presented in Sections \ref{subs:screw}
and \ref{subs:construction}. Applications to the value function of a discounted infinite horizon optimal control problems,
including its characterization as the unique continuous constrained 
solution of the corresponding Hamilton-Jacobi-Bellman equation are described in Section \ref{sec:applications}.

\section{Notation and basic definitions}\label{sec:notat}
In this manuscript, $\mathbb{B}$ denotes the closed unit ball in $\mathbb{R}^n$. 
Given a subset $X$ of $\mathbb{R}^{n}$, $\partial X$ is the boundary of $X$. 
The Euclidean distance of a point $y$ from the set $X$ (that is $\inf \limits_{x \in X} \| x-y \| $) is written 
$d_X(y)$, and the 
$\delta$-neighborhood of $X$ is the open set $X_\delta:=\{ x\in \mathbb{R}^n: d_X(x)< \delta\}$, $\delta >0$.

For a given scalar function $g:\R^n \to \R$, its gradient at some point $x$ is denoted by $\nabla g(x)$. 
When $g$ is a vector-valued function, we use $Dg(x)$ for its derivative and $D^2 g(x)$ for its hessian.

The {\it Lie bracket } $[g_1,g_2]$ of two $C^1$ vectors fields $g_1,g_2: \R^n \to \R^n$ is the vector field defined as  $[g_1,g_2]:= Dg_2g_1 - Dg_1 g_2$\footnote{As is well-known the Lie bracket is an intrinsic object, i.e. it can be defined independently of coordinates}.
It satisfies the anti-symmetry relation $[g_1,g_2]\equiv - [g_2,g_1]$.

If  $T>0$ and $\alpha:[0,T]\to\R^n$ is an $L^\infty$ arc, we use $ \|\alpha\|_{L^\infty}$ to denote 
the $L^\infty$ norm. 
Furthermore, if $x$ is { absolutely continuous} \, namely  class $x\in W^{1,1}(I)$, one sets
\[
\|x\|_{W^{1,1}}:=\| x(0)\|+ \|\dot{x}\|_{L^\infty}.
\]
For the general dynamics \eqref{gen_dyn} (with $f$ smooth enough), the trajectory $x(.)$ corresponding to the control $u(.)$ such that
$x(0)=\xi$ is denoted by $x_\xi^u(.)$. The couple $(x,u)$, where $x(.)$ is the solution of \eqref{gen_dyn} corresponding to the control $u(.)$ is called a \textit{process}. A process $(x,u)$ is said \textit{feasible} (with respect to the closed set $C$)
provided it satisfies the state constraint $x(t)\in C$ for all $t\in [0,T]$.

Consider now the problem of minimizing a functional $J$ depending on trajectories of \eqref{gen_dyn}. We say that a process
$(x,u)$ is a $W^{1,1}$-minimizer for $J$ subject to \eqref{gen_dyn}
if there exists $\delta >0$ such that for all processes $(y,v)$ of \eqref{gen_dyn} with $\|y-x\|_{W^{1,1}(0,T)}<\delta$ one has $J(x)\le J(y)$.

\subsection{Quadratic forms}\label{sec:QF} We will make use of the  trivial  one-to-one correspondence between quadratic forms $\mathbf Q$ on $\R^2$   and  real $2\times 2$ symmetric matrix $Q_{\ell m}$, defined by  the validity of the relation   $\displaystyle U\mapsto{\mathbf Q}(U) = \sum_{l,m=1}^2 Q_{l,m}U^lU^m$ for all $(U_1,U_2)\in\R^2$. \footnote{ Therefore, when the context is unequivocal, we sometimes will use the term symmetric matrix and quadratic form indifferently.}
\begin{definition}We say that a quadratic form $\mathbf{Q}$ is {\rm non positive} if $\mathbf{Q}$ is indefinite,  
		--i.e. $\Delta:= Q_{11}Q_{22}-Q_{12}^2<0$,-- or  negative semidefinite, 
		--i.e. $\Delta= 0$ and $Q_{11},Q_{22}\leq 0$,-- or negative definite,  
		--i.e., $\Delta >0$ and $Q_{11} <0$ $(\iff  Q_{22} <0 )$.
	\end{definition}
If the quadratic form $\mathbf{Q}$ is indefinite, or negative semidefinite and nonvanishing, let 
	$(\cos \varphi_\mathbf{Q},\sin \varphi_\mathbf{Q})$, with $\varphi_{\mathbf Q}\in [-\pi,\pi[$, 
	be the direction of the eigenvector 
	corresponding to the negative eigenvalue of $\mathbf{Q}$ and let $\beta_{\mathbf Q}\in ]0,\frac{\pi}{2}]$ be such that
	$\big(\cos (\varphi_\mathbf{Q}-\beta_\mathbf{Q}),\sin (\varphi_\mathbf{Q}-\beta_\mathbf{Q})\big)$ and
	$\big(\cos (\varphi_\mathbf{Q}+\beta_\mathbf{Q}),\sin (\varphi_\mathbf{Q}+\beta_\mathbf{Q})\big)$ are 
	the two directions along with $\mathbf{Q}$ vanishes. Observe that the limit case 
	$\beta_{\mathbf{Q}}=\frac{\pi}{2}$ corresponds to $\mathbf{Q}$ negative semidefinite. Moreover,
	for each $U\neq (0,0)$ it holds ${\mathbf Q}(U)<0$ if and only if the argument $\vartheta$ of $U$ (or of $-U$) belongs
	to the open interval $]\varphi_{\mathbf Q}-\beta_{\mathbf Q},\varphi_{\mathbf Q}+\beta_{\mathbf Q}[$.
	Given the quadratic form ${\mathbf Q}$ as above,
	\begin{equation}\label{formaQ}
		\text{we call $\varphi_\mathbf{Q}$ and $\beta_\mathbf{Q}$ the {\it principal  direction} and 
			{\it half-amplitude} of $\mathbf{Q}$.}
\end{equation}
Let $f_1,f_2:\mathbb{R}^n\to \mathbb{R}^n$ be differentiable and let $h:\mathbb{R}^n\to \mathbb{R}$ be twice
	differentiable.
	For every $x\in \R^2$, the $2\times 2$ matrix  ${\Lambda}(x) = 	\Big(\Lambda_{\ell m}(x)\Big)_{\ell,m=1,2}$, defined by
		\begin{equation} \label{equation: S+A italy}
		\Lambda_{\ell m}(x) := f_\ell (x) \cdot D^2 h(x) f_m(x) + \nabla h(x) \cdot  Df_\ell(x) f_m (x) \qquad \forall \ell,m=1,2,
	\end{equation}
	will be used throughout the paper.
	We will consider, in particular, the {\it symmetric part $S(x)$}   and the {\it antisymmetric part  $A(x)$}  of ${\Lambda}(x)$ defined by 
	\begin{align}
		S_{\ell m}(x) := \frac 12 (\Lambda_{\ell m}+\Lambda_{m \ell})&=\frac{1}{2} \nabla h(x) \cdot \Big( Df_\ell(x) f_m (x)+Df_m(x) f_\ell (x)\Big)+ 
		f_\ell (x) \cdot D^2 h(x) f_m(x) \label{def_Slm}, \\
		A_{\ell m}(x)  := \frac 12 (\Lambda_{\ell m}-\Lambda_{m \ell})&=  \frac{1}{2} \nabla h(x) \cdot \Big( Df_\ell(x) f_m (x)-Df_m(x) f_\ell (x)\Big) = \frac{1}{2}\nabla h(x) 
		\cdot [f_m,f_\ell] (x).\label{def_Alm}
\end{align}
The following result  states  that  each entry of  the matrix ${ S}(x)$ is independent of the choice of coordinates $x$, in other words each of such entries is an invariant.  This suggests that the main result of the paper might be extended to the case where the state $x$ belongs to a differential manifold.  
	\begin{proposition}\label{inv_S}
		The matrix ${S}(x)$ is independent of local coordinate changes. Namely for any local diffeomorphism
		$\tilde{x}=\tilde{x}(x)$ defined on some open subset $O\subseteq\R^n$ one has
		\[
		\tilde{S}(\tilde{x}(x)) = {S}(x)\;\text{ for each $x\in O$,}
		\]	where $\tilde{S}$ is constructed as $S$ but in the coordinates $\tilde{x}$.
	\end{proposition}
The proof of the above result is postponed to the Appendix.
\subsection{Drops}
	\begin{definition}	Given a locally absolutely continuous  curve $R=(R_1,R_2):\mathbb{R}\to \mathbb{R}^2 $ 
and  $t_1,t_2\in\mathbb{R},$ $ t_1<t_2$, motivated by Gauss-Green formula, let us define
$\mathbf{Area}(R_1,R_2)|_{t_1}^{t_2}$, the {\rm area swept} by $R$ on 
$[t_1,t_2]$,  by means of the formula: 
	\[
	\mathbf{Area}(R)|_{t_1}^{t_2}:=\frac12 \int_{t_1}^{t_2} \left[\frac{\dd R_2}{\dd s}(s) R_1(s) - \frac{\dd R_1}{\dd s}(s) R_2(s)\right] \dd s.
	\]
Moreover, let us define $\mathbf{Exc}(R_1,R_2)|_{t_1}^{t_2}$,  the {\rm excess} swept by $(R_1,R_2)$ on $[t_1,t_2]$,
by setting
	\[
	\mathbf{Exc}(R)|_{t_1}^{t_2}:= 
	\sum_{\ell,m=1,2}\int_{t_1}^{t_2}\!\!\!\!\int_0^s \bigg|\xi\frac{\dd R_\ell(\xi)}{\dd \xi}\dd\xi\bigg|\, 
	\bigg|\frac{\dd R_m(s)}{\dd s}\bigg|\dd s.
	\]
\end{definition}
\begin{remark}\label{rempar}
	{\rm Let us observe that both the area and the excess are intrinsic quantities, i.e., they are independent of the curve's parameterization.}
\end{remark}

\begin{definition}\label{drop}	
We call {\rm drop of phase} $\varphi\in [-\frac{\pi}{2},\frac{\pi}{2})$, {\rm half-amplitude} 
$\beta\in (0,\frac{\pi}{2}]$, and {\rm period} $\tau_0 >0$ any simple curve
		\(R=(R_1,R_2):\mathbb{R}\to \mathbb{R}^2 \) verifying the following properties:	
		\begin{enumerate}
			\item $R$ is parameterized with arc-length, i.e., $\left|\frac{\dd R}{\dd t}(t)\right|=1$ for almost every $t\in\mathbb{R}$;
			\item $R(0)=0$ and $\frac{\dd R}{\dd t}(0)= (\cos(\varphi-\beta),\sin (\varphi-\beta) )$;
			\item $R$ is $\tau_0$-periodic and $\frac{dR}{dt}$ is  $L$-Lipschitz continuous, for some $L>0$;
			\item the  area $\sigma\mapsto \mathbf{Area}(R)_{0}^{\sigma}$ is strictly increasing and there exists a positive constant  $C_R$ such that 
			\[
			\mathbf{Exc}(R)_0^{\sigma} \leq C_R\big|\mathbf{Area}(R)_0^{\sigma}\big|,\quad\forall \sigma\in[0, \tau_0];\]
			\item one has \[R(]0,\tau_0[) \subset 
 \{ U\in\mathbb{R}^2 : \varphi - \beta < \mathrm{Arg}(U) < \varphi + \beta\}
\]
(here $ \mathrm{Arg}(U)$ is the unique angle in $]-\pi,\pi]$ such that $U=|U|e^{i\, \mathrm{Arg}(U)}$, $U\neq 0$).
		\end{enumerate}	
When $\beta<\pi/2$, we say that the the drop $R$  is {\rm pointed}. Instead,  if a drop  
$R$ has  half-amplitude $\beta=\pi/2$,  we say that  $R$ is  {\rm circle-like}.
\end{definition} 
Observe that simple estimates for both $\mathbf{Exc}(R)_0^{t}$ and $\mathbf{Area}(R)_0^{t}$ can be obtained
immediately from the Lipschitz continuity of $\dot{R}$. Indeed,
\begin{equation}   \label{rough_est_exc}
\mathbf{Exc}(R)_0^{t} \le \frac23 L^2 t^3.
\end{equation}
Moreover, since, for $i=1,2$ we can write $R_i(t)=\dot{R}_i(0)t+\int_0^t\!\!\int_0^s\ddot{R}_i(\xi)\dd\xi\dd s$, it holds
\begin{equation}\label{rough_est_area}
\begin{split}
\mathbf{Area}(R)_0^{t}&=\frac12 \int_0^t\bigg(-\Big( \dot{R}_2(0)s+\int_0^s\int_0^\sigma \ddot{R}_2(\xi)\dd\xi\Big)
   \Big(\dot{R}_1(0)+\int_0^s\ddot{R}_1(\xi)\dd\xi\Big) \\
   &\qquad \qquad +\Big( \dot{R}_1(0)s+\int_0^s\int_0^\sigma \ddot{R}_1(\xi)\dd\xi\Big)
   \Big(\dot{R}_2(0)+\int_0^s\ddot{R}_2(\xi)\dd\xi\Big)\bigg)\dd s\\
   &= O(t^3)\quad\text{for }t\to0^+.
\end{split}
\end{equation}
\begin{definition} If for some drop $R$, the curve $r:\mathbb{R}\to\mathbb{R}^2 $ is defined as 
$r(t) = \frac{dR}{dt}(t)$ for almost all $t\in\mathbb{R}$, we say that $r$ is a {\rm drop-generator} and 
we call $R$ the {\rm primitive drop} of $r$. A drop generator $r$ will be said to have {\rm phase} $\varphi$ and
{\rm half-amplitude} $\beta$ if its primitive drop $R$ has phase $\varphi$ and half-amplitude $\beta$.  
Similarly, we say that a   drop generator $ r$ is  {\rm pointed}  
[resp. {\rm circle-like}], if and only if its primitive $R$ is  
{\rm pointed} [resp. {\rm circle-like}].\end{definition} 
Observe that for any $\tau_0$-periodic drop generator $r$ the condition $\int_0^{\tau_0} r(s)\dd s =0$ holds.

\medskip

Finally, we will be interested in the sign of quadratic forms on the image of drops.

\begin{definition} We say that a drop $R$ (and its generator) of phase $\varphi$ and  half-amplitude  $\beta$ is
{\rm adapted to a non-positive symmetric quadratic form $\mathbf{S}$} if 
\[
\varphi=\pm\varphi_{\mathbf{S}} \quad \text{and}\quad \beta\le\beta_{\mathbf{S}}.
\]
\end{definition}
\begin{remark}{\rm Notice, in particular, that if a drop $R$ is adapted to a non-positive symmetric quadratic form $\mathbf{S}$,
then 
\begin{equation}\label{S<0}
\mathbf{S}(R(t)) \leq 0 \qquad\forall t\in\mathbb{R}.
\end{equation}
Let us observe that if $R$ is circle-like and is adapted to $\mathbf{S}$, then necessarily the quadratic form $\mathbf{S}$ is 
negative semidefinite. An example of an adapted pointed drop is depicted in Figure \ref{fig:ad_drop}.}
\end{remark}
\begin{figure}
\centering
\includegraphics[width=7.2truecm]{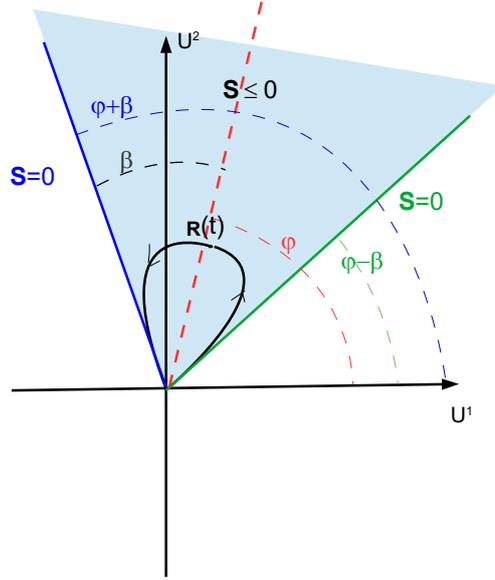}
\caption{An adapted drop}
\label{fig:ad_drop}
\end{figure}
\subsection{Examples of drops adapted to non-positive quadratic forms} \label{constr_drops}
\subsubsection{A circle-like drop generator} 
An obvious example of circle-like drop generator of period $\tau_0 >0$ is the circle
$t\mapsto r(t):=\left(\begin{matrix}\cos (2\pi t/\tau_0-\pi/2)\\\sin(2\pi t/\tau_0 -\pi/2)\end{matrix}
\right)$. 
Clearly, $r$ has phase $0$ (and half-amplitude $\pi/2$), while a rotation of $r$ by an angle $\varphi \in ]-\pi/2,\pi/2]$ 
yields any desired phase. Observe that, for $t\in [0,\tau_0]$, the area
generated by $R(t)=\int_0^t r(s)\dd s$ is
\[
\mathbf{Area}(R)_0^{t} = \tau_0\, \frac{2 \pi t - \tau_0 \sin\frac{2 \pi t}{\tau_0}}{8 \pi^2}  \ge \frac{1}{2\pi^2} \, \frac{t^3}{\tau_0},
\]
since $\frac{2\pi t}{\tau_0}-\sin \frac{2\pi t}{\tau_0}\ge 4\frac{t^3}{\tau_0^3}$ for $t\in [0,\tau_0]$.
Therefore, taking into account \eqref{rough_est_exc}, in order to obtain (4) of Definition \ref{drop} it is enough
to take $\tau_0$ small enough. All other requirements of Definition \ref{drop} are fulfilled. Similar computations can be
performed for ellipses of half amplitudes $A,B>0$, by using arclength.

\subsubsection{A pointed drop}\label{pted_drop}
 Let $\beta\in (0,\frac{\pi}{2}]$ and $P>0$ be given and consider the circular arc
\[
R(t)= \begin{pmatrix}R_1(t)\\
R_2(t)
\end{pmatrix}
:= \begin{pmatrix}\frac{P}{\pi} \cos \big(\frac{\pi t}{P}-\beta\big) \sin\frac{\pi t}{P}\\
\frac{P}{\pi} \sin \big(\frac{\pi t}{P}-\beta\big) \sin\frac{\pi t}{P}
\end{pmatrix}
=\begin{pmatrix}
\int_0^t \cos\big(\frac{2\pi s}{P}-\beta\big)\dd s\\
\int_0^t \sin\big(\frac{2\pi s}{P}-\beta\big)\dd s
\end{pmatrix}.
\]
Setting $P_\beta:=\displaystyle\frac{P\beta}{\pi}$, we observe that 
$R(P_\beta)=\Big(P\displaystyle{\frac{\sin \beta}{\pi}},0\Big)$. 
Moreover, $R$ has phase $0$ and half amplitude $\beta$. 
For $P_\beta\le t \le 2 P_\beta$, let $R$ be the reflection of the above curve with respect to the $x$-axis, 
traveled backwards, i.e., for $t\in [P_\beta, 2P_\beta]$, $R(t):=\big(R_1(t-P_\beta),-R_2(t-P_\beta)\big)$. 
Finally we extend $R$ on the whole of $\R$ by
$2P_\beta$-periodicity. Observe now that, for $0\le t \le P_\beta$ the area spanned by $R$ behaves
exactly as in the previous example, so that for $P$ small enough estimate (4) of Definition \ref{drop} holds.
For $P_\beta\le t \le 2 P_\beta$, by symmetry, we have, for $P$ small enough,
\begin{align*}
\mathbf{Area}(R)_0^t &=\mathbf{Area}(R)_0^{P_\beta}+\mathbf{Area}(R)_{P_\beta}^t\ge
\shortintertext{(arguing as in the circle-like case)}
&\ge \frac{1}{2\pi^2} \bigg( \frac{(t-P_\beta)^3}{P^2}+\frac{P_\beta^3}{P^2} \bigg)\\
&\ge \frac{1}{8\pi^2}\frac{t^3}{P^2},
\end{align*}
so that (4) of Definition \ref{drop} remains valid. A rotation of
$R$ by an angle $\varphi$ yields a curve of phase $\varphi$ and half amplitude $\beta$. 
Finally, observe that $R$ has period $2P_\beta=2\frac{P\beta}{\pi}$. Therefore,
by choosing $P=\frac{\pi}{2\beta}\tau_0$ a curve $R$ with preassigned period $\tau_0$ is constructed. 
Observe that
the choice of $P$ is proportional to $\tau_0$.
The case $\varphi=0$, $\beta=\frac{\pi}{4}$ is depicted in Figure \ref{fig:pted_drop}.
\begin{figure}
\centering
\includegraphics[width=8.5truecm]{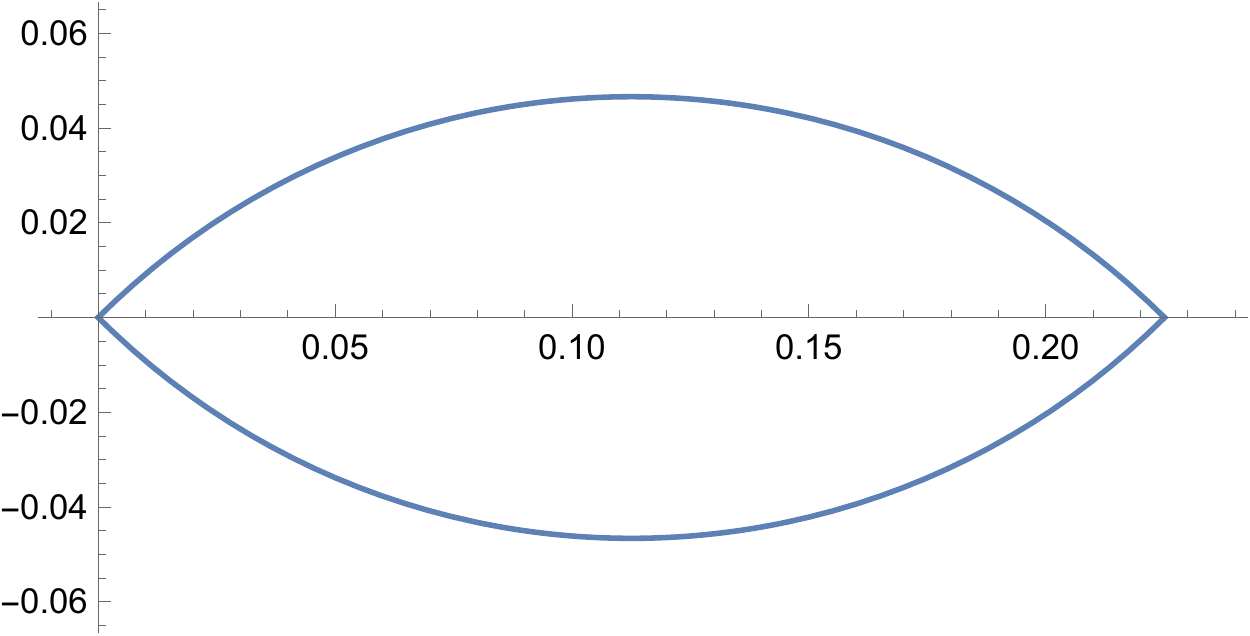}
\caption{The pointed drop constructed in Sect. \ref{pted_drop}}
\label{fig:pted_drop}
\end{figure} 
 	
\subsection{Rotational controls}	\label{sec:rot_contr}

For some $\tau_0 ,L>0$, let $r$ be a $\tau_0$-periodic, $L$-Lipschitz continuous drop generator, and, for any 
{\it angular velocity} $\omega \in\mathbb{R}$, let us  consider the control 
		\begin{equation}\label{eq:rot_controls}
		u_\omega(t) := r(\omega t),\quad \forall  t\in\mathbb{R},
	\end{equation}	
	which will be called a {\it rotational control of base $r$ and angular velocity $\omega$}.
	The trajectory verifying (\ref{control system}) that corresponds to  the control $u$ will be termed 
	$y_{\omega,x_0}
	$. (The dependence on  $\omega$ will be often omitted and we shall write $u$ and $y_{x_0}$ (or even $y$) instead of $	u_{\omega}$ and
	$y_{\omega,x_0}$, respectively). 
	
	Let us  consider the {\it first iterated integrals }
	\begin{align}
		t\mapsto(U^1_\omega,U^2_\omega)(t) :=\int_0^t\left( u^1_\omega , u^2_\omega \right)(s)  \dd s. \end{align} 
	Notice that, for every $\omega \in\mathbb{R}$ , {\it $(U^1_\omega,U^2_\omega)$
		is $\tau_0/\omega$-periodic}, by definition of drop generator.	
	Furthermore,
	\[|(U^1_\omega,U^2_\omega)(t)|\leq \frac{L\tau_0}{\omega}, \,\,\,\,\forall t\in\mathbb{R}.\]
	With a change of variable we also get
	\begin{equation}\label{wt}
		\left(	U^1_{\rho\omega},U^2_{\rho\omega}\right)(t) = \frac{1}{\rho}	\left( U^1_{\omega}(\rho t), U^2_{\omega}(\rho t)\right) \qquad\forall  t,\omega,\rho\in\mathbb{R},\,\,\rho\neq 0.\end{equation} 
For every $\ell,m=1,2$, let us also consider the {\it second  iterated integrals} 	
\[
	U^{\ell m}_\omega(t) := \int_0^t u^\ell_\omega(s) U^m_\omega(s)\dd s,
\]
together with their {\it symmetric} and {\it antisymmetric part}, namely
\[
U_{S,\omega}^{\ell m}(t):=\frac 12(U^{\ell m}_\omega(t) +U^{m\ell}_\omega(t))
\]
and  
\[U_{A,\omega}^{\ell m}(t):=\frac12 (U_\omega^{\ell m}(t) - U_\omega^{m\ell}(t)),
\]
	respectively. 

One easily checks that, for all $t\in\mathbb{R}$ and $\ell,m=1,2$, 
	\begin{equation}\label{U ell m}
		U_{S,\omega}^{\ell m}(t)= \frac 12 U_\omega^\ell(t)U_\omega^m(t), \end{equation} and, furthermore,	
	\[ U_{A,\omega}^{11}(t)= U_{A,\omega}^{22}(t)\equiv 0 , \qquad U_{A,\omega}^{21}(t)=-U_{A,\omega}^{12}(t)=\mathbf{Area}(U^1_\omega,U^2_\omega)|_0^t.
	\]
Notice that, by \eqref{wt} and \eqref{U ell m}, one  has
	\[U_{S,\omega}^{\ell m}(t) = \frac{1}{\omega^2}R^{\ell m }_S(\omega t)\qquad \forall \ell,m =1,2,\,\,\,\forall t,\omega\in \mathbb{R}, \,\,\
	\]
and
\begin{equation}\label{area om}
U_{A,\omega}^{21}(t)=\mathbf{Area}(U^1_\omega,U^2_\omega)|_0^t=\frac{1}{\omega^2}\mathbf{Area}(R^1,R^2)|_0^{\omega t},
\end{equation}
where $R=(R^1,R^2)$ is the primitive drop of $r=(r^1,r^2)$ and the $R^{\ell m}$, $\ell,m=1,2$,
are defined with respect to $r$ analogously to $U^{\ell m}_\omega$,  with respect to $u_\omega$, namely
$R^{\ell m}(t):=\int_0^t r^\ell(s) R^m(s)\dd s$.

\section{Assumptions and statement of the main results}\label{sec:ass}
Let  $f_1, f_2: \R^n \to \R^n$ be vector fields, let $h: \R^n \to \R$ be a continuous function, 
and consider the closed subset 
\begin{equation} \label{state constraint} C := \{ x : h(x) \le 0  \} .  \end{equation}  
We will be concerned with  the nonlinear control system 
\begin{equation} \label{control system}  \dot{x}  = u_1 f_1(x) + u_2 f_2(x), \quad x(0)=x_0, \end{equation}
subject to the state-space constraint 
\begin{equation}\label{state constraint condition}
x(t)\in C,\qquad t\in [0,T].
\end{equation}
Here, $x$ and $u=(u_1,u_2)$ are the state and the control, respectively, and $x_0$ is the initial condition. 
A trajectory is termed {\it feasible} if it satisfies both (\ref{control system}) and (\ref{state constraint condition}).

\medskip

The data enjoy the following basic assumptions, together with further ones that will be introduced later:
\begin{itemize}
	\item[\textbf{(H1)}] $f_1, \ f_2$ are of class $C^{2}$, Lispchitz with common Lipschitz constant $k_f$, and for all $x\in \R^n$, $|f_1(x)| + |f_2(x)| \le k_0$ for some $k_0>0$.
	\item[\textbf{(H2)}] $(u_1, u_2) \in \mathbb{B} := \{ (u_1,u_2) \in \R^2 \ : \  u_1^2+ u_2^2 \le 1\}$.
	\item[\textbf{(H3)}]  $h$ is of class $C^{2,1}$ and is Lipschitz with constant $k_h$,
$C$ is nonempty,  and $\nabla h(x) \neq 0$ for all $x$ such that $h(x)=0$, so that, in particular, $C$ has nonempty interior.
\end{itemize}
We define the \textit{singular set} $\mathbb{S}$ as the set where the classical first order 
Inward Pointing Condition \eqref{classical_IPC} is violated:
\begin{equation}\label{eq:defS}
\mathbb{S} := \{ x \in C \ : \  \nabla h(x) \cdot f_1(x) = \nabla h(x) \cdot f_2(x) = 0  \}.
\footnote{Observe that $\mathbb{S}$ must have empty interior (see, e.g., \cite[Remark 2.2]{motta}).}
\end{equation}
We will make the following assumptions, that involve $\mathbb{S}$:
\begin{itemize}
	\item[\textbf{(H4)}] there exists $\alpha > 0$ such that, for all $x\in\mathbb{S}$, 
$\big|\nabla h(x) \cdot [f_1,f_2](x)\big| > \alpha$.
	\item[\textbf{(H5)}]  For any $x\in\mathbb{S}$, recalling the quadratic form $\mathbf{S}(x)$ that was defined in Sect. \ref{sec:QF},
one of the following properties is verified:
\begin{itemize} \item[\textbf{(H5)(i)}] the  form $\mathbf{S}(x)$ is negative semidefinite; 
\item[\textbf{(H5)(ii)}]  the form $\mathbf{S}(x)$ is indefinite. 
\end{itemize}
\item[\textbf{(H6)}] there exist $\delta> 0$ and $d_0>0$ such that
	\begin{equation}\label{ass_H6}
	 \sqrt{(\nabla h(x) \cdot f_1(x))^2 +  (\nabla h(x) \cdot f_2(x))^2} \ge d_0 \,d_{\mathbb{S}}(x) \; \text{  for all } x \in \mathbb{S} + \delta \B. 
	 \end{equation}
\end{itemize}
\begin{remark}{\rm
Notice that, since the control set $\mathbb{B}$ is the unit ball of $\R^2$, \textbf{(H6)} is equivalent to 
\begin{itemize}
\item[\textbf{(H6)'}] there exist $\delta> 0$ and $d_0>0$ such that
\[\min_{(u_1,u_2)\in U}\nabla h(x) \cdot\left( f_1(x)u_1 +   f_2(x)u_2\right) \leq - d_0 \,d_{\mathbb{S}}(x) \; \text{  for all } x \in \mathbb{S} + \delta \B.\]
\end{itemize}}
\end{remark}
\begin{remark} {\rm Without loss of generality, the right hand side of \eqref{control system} 
can be supposed to be fully nonlinear, i.e., 
equal to $f(x,u)$, out of a neighborhood of the singular set $\mathbb{S}$, while in a neighborhood of $\mathbb{S}$
the right hand side of \eqref{control system} can be supposed to be of the form
\[
\sum_{i=1}^m f_i(x) u_i,
\]
with $m\ge 2$ and such that two among the vector fields $f_i$ satisfy the previous assumptions.
Moreover the control set $\mathbb B$ can be supposed to be any
set in $\mathbb{R}^m$ that contains the origin in its interior.}
\end{remark}

\medskip

The key result of the paper involves a construction of neighboring feasible trajectories together with an estimate on the distance from the reference one. Before stating the result, we introduce some notation that will be used along this section. Let 
$(\bar x(.), \bar u(.))$ be a reference process of (\ref{control system}) such that $h(x_0)=0$ (where $h$ is as in (\ref{state constraint})), that possibly does not satisfy the constraint $\bar{x}(t)\in C$ for all $t$ in a given time interval $[0,T]$. 
Let $d$ be the maximum extent of the violation of the state constraint condition of $(\bar x(.),\bar u(.))$ defined on $[0,T]$:
\begin{equation}\label{eq: expression of the violating extent} 
d:= \max\limits_{t\in [0,T]} \{ \max \{ h(\bar x(t)), 0 \}   \}  .  
\end{equation}
The first time after which the reference trajectory violates the state constraint will be denoted by $\tau_1$, which is defined as follows:
\begin{equation}\label{eq: expression of the violating time}
\tau_1 := \begin{cases}  T \quad & \text{if } d=0 \\ \inf \{ t \in (0,T] \ : \ h(\bar x(t)) > 0 \} \quad  & \text{if } d>0 .
\end{cases}
\end{equation}
We can now state our result on neighboring feasible trajectories.
\begin{theorem}\label{theorem:NFT}
	Assume \textbf{(H1)}-\textbf{(H6)} and let $(\bar x, \bar u)$ be a reference process on the interval $[0,T]$,
with initial condition ${x}_0$ such that $h(x_0)\le 0$. 
Then for all $d$ small enough there exists another process $(y(.),u(.))$ of \eqref{control system} 
that enjoys the following properties:
	\begin{align}\label{estimate: on first interval}
	 y(t)&=\bar{x}(t) \text{ for all } t\in [0,\tau_1] ,\\
	\label{estimate: on second interval}
	y(t)\in \overset{\circ}{C},  \,\,\hbox{\text i.e.,} \,\,  h(y(t))& < 0 \text{ for all } t\in (\tau_1,T] ,\\
\label{eq:nonlinest} 
 	    \big\| y(.)-\bar x(.)\big\|_{L^\infty(0,T)}  &\leq K\sqrt{d},\\
\label{contrest}
\| u(.) - \bar u(.)\|_{L^1(0,T)}&\le K \Big(\sqrt{d} + \big\| \bar u(.-\tau(d))-\bar{u}(.)\big\|_{L^1(\tau_1+\tau (d),T)}\Big),\\
\label{W11est}
\big\| \dot{y}(.)-\dot{\bar x}(.)\big\|_{L^1(0,T)}&\le K\Big(\sqrt{d} + 
\big\| \dot{\bar x}(.-\tau(d))-\dot{\bar{x}}(.)\big\|_{L^1(\tau_1+\tau (d),T)}\Big),
\end{align}
for a suitable constant $K$ depending only on $h,f_1,f_2, T$ and a suitable nonnegative function  
$\tau(\cdot)$ verifying $\tau(d)\le K\sqrt{d}$. Therefore, in particular,
\[
\big\| \dot{y}(.)-\dot{\bar x}(.)\big\|_{L^1(0,T)}\to 0\;\text{ as $d\to 0$}
\]
and
\[
\big\| y(.)-\bar{x}(.)\big\|_{W^{1,1}(0,T)}\le K'\sqrt{d},
\]
for a suitable constant $K'$ depending only on $h,f_1,f_2, T$, 
provided $\dot{\bar x}$ has bounded variation (that holds, in particular, if $\bar u$ has bounded variation).
\end{theorem}     

\section{Examples}\label{sec:examples}
\subsection{Counterexamples}
We recall first an example which is essentially known, and shows that the failure of the first order inward pointing
condition may not allow to construct neighboring feasible trajectories. 
In $\mathbb{R}^2$, let $f(x_1,x_2)=(1,0)$ and consider the dynamics
\begin{equation}\label{counterex}
(\dot{x}_1,\dot{x}_2)=f(x_1,x_2)u,\quad u\in [-1,1],
\end{equation}
with the constraint
\begin{equation}\label{constr_counterex}
C= \big\{ (x_1,x_2): x_2\le x_1^2 \big\}.
\end{equation}
For $\varepsilon > 0$, define the family of trajectories
\[
\gamma_\varepsilon (t) := (-\varepsilon,\varepsilon^2)+ (t,0),\quad t\in [0,1].
\]
The neighboring feasible trajectory that is the closest to $\gamma_\varepsilon$ is $y_\varepsilon:\equiv
(-\varepsilon,\varepsilon^2)$, while the extent $d$ of the deviation 
of $\gamma_\varepsilon$ from the constraint, as defined in \eqref{eq: expression of the violating extent}, is $\varepsilon^2$. 
Therefore no estimate of the distance between
$y_\varepsilon$ and $\gamma_\varepsilon$ that vanishes as $\varepsilon \to 0$ can be obtained. 
Setting $g(\xi,\eta)=\frac12 (\eta-1)^2$, we see also that the value function $v$ of the problem of minimizing $g$ along all solutions
of \eqref{counterex} belonging to $C$ for all $t\in [0,1]$ is not upper semicontinuous at $\xi=(0,0)$. 
The analogous Example IV.5.3
in \cite{BCD} shows that the value function of a suitable infinite horizon discounted problem (with a Lipschitz Lagrangean) under
the dynamics \eqref{counterex} subject to the constraint \eqref{constr_counterex} fails as well to be continuous at $(0,0)$.
The first order IPC condition, as well as \textbf{(H4)}, are obviously violated.
\subsection{Positive examples and sharpness of the assumption \textbf{(H3)}}
Consider Brockett's non-holonomic integrator
\begin{equation}\label{eq:syst}\dot x = f_1(x)u_1 + f_2(x)u_2, 
\end{equation}
with
\begin{equation}\label{nonhol}
f_1 = \left(
\begin{array}{c}
1\\
0\\ -x_2
\end{array}
\right) \ , \quad \quad f_2 = \left(
\begin{array}{c}
0\\
1\\ x_1
\end{array}
\right),
\end{equation}
subject to the flat state constraint 
\[   h(x) := x_3 \le 0   , \quad  \text{for all } x \qquad \text{for }  h: \R^3 \to \R. \]
Observe that the vector fields $f_1,f_2$ involved in this example fail to satisfy the inward 
pointing condition \eqref{classical_IPC} at any point $y$ belonging to the $x_3$-axis. On the contrary, 
the Lie Bracket of $f_1$ and $f_2$ at any point $x\in \R^3$ is
\[  [f_1,f_2](x) = \left(
\begin{array}{c}
0\\
0\\ 2
\end{array}
\right)  , \] 
and verifies
\[[f_1,f_2](x) \notin \big\{ u_1f_1(x)+u_2f_2(x) \text{ for all } (u_1,u_2) \text{ such that } u_1^2+u_2^2 \le 1  \big\}, \qquad \text{ for all } x\in \R^3\]
together with
\[   \nabla h(x)\cdot [f_1,f_2](x) > 0 \  \text{ for all } x\in \R^3   .\] 
Moreover, all entries of the matrix $S$ are easily seen to be zero at all points in $\mathbb{R}^3$, so that \textbf{(H5)} is verified. 
Finally, $(\nabla h(x)\cdot f_1(x))^2+(\nabla h(x)\cdot f_2(x))^2=x_1^2+x_2^2$, so that \textbf{(H6)} holds as well, with $d_0=1$.

The more general state constraint
\[
h(x_1,x_2,x_3):= \lambda (x_1^2+x_2^2)^p +x_3 \le 0,\; \lambda > 0,
\]
still satisfies all the assumptions \textbf{(H1)}-\textbf{(H6)}, provided $p\ge \frac32$, for the vector fields \eqref{nonhol}. 
Indeed, $h$ is of class $\mathcal{C}^{2,1}$ for all $p\ge \frac32$,
the singular set $\mathbb{S}$ still coincides with the $x_3$-axis, and $\nabla h(x)\cdot [f_1,f_2]=2$ on it, 
and the quadratic form $S$ is easily seen to be vanishing on $\mathbb{S}$, provided $p>1$. Finally,
$\big((\nabla h(x)\cdot f_1(x))^2+(\nabla h(x)\cdot f_2(x))^2\big)/(x_1^2+x_2^2)$ is easily seen to be
bounded away from zero as $(x_1,x_2)\to (0,0)$, provided again $p>1$. 

Observe that Brockett's vector fields allow explicit computations:
the exact solution of the Cauchy problem
\[ \dot x= f_1u_1+f_2u_2 \qquad x(0)=(0,0,0) \]
for the particular choice of rotational controls $(u_1(t),u_2(t))=(\cos \omega t, \sin \omega t)$
---indeed, the rotational strategy that we will employ in the construction of Sect. \ref{subs:screw}
was inspired by this example--- is
\begin{align*}
& x_1(t) = \frac{ \sin \omega t}{\omega }, \qquad x_2(t) = \frac{1 -  \cos \omega t}{\omega }, \qquad x_3(t) = \frac{\omega t- \sin \omega t}{\omega^2}.
\end{align*}
The scalar function $h$ at the computed trajectory $x(.)$ is:
\[ h(x(t)) =  2^p \lambda  \left( \frac{1-\cos (\omega t)}{\omega^2} \right)^{p} +\frac{ \omega t - \sin \omega t  }{\omega^2} . \]
A series expansion around $t=0$ shows that
\[
h(x(t))=\frac{\omega}{6}t^3+o\big(t^3\big) + \lambda t^{2p}+o\big(t^{2p}\big).
\]
Hence, for $p<\frac32$ the constraint $h(x(t))\le 0$ is not satisfied at all $t>0$ small enough, independently of the choice of $\omega$, 
while if $p\ge \frac32$ it is satisfied for all $t$ small enough if $\omega < - 6\lambda$. 
This shows that the assumption of $\mathcal{C}^{2,1}$-regularity on $h$ is
sharp for the construction that will allow us to prove Theorem \ref{theorem:NFT}. 

More in general, consider the system \eqref{eq:syst} with $f_1$, $f_2$ given by
\[
f_1(x,y,z) = \left(
\begin{array}{c}
\gamma_1\\
0\\ -x_2
\end{array}
\right) \ , \quad \quad f_2(x,y,z) = \left(
\begin{array}{c}
0\\
\gamma_2\\ x_1
\end{array}
\right)
\]
together with the constraint $\{ (x,y,z): h(x,y,z)\le 0\}$, with
\[
h(x,y,z):= \lambda (c_1 x^2+c_2y^2)^p +z.
\]
Straightforward computations show that the singular set $\mathbb S$ is $\{ (0,0,z): z\in\mathbb R\}$ and that condition \textbf{ (H6)}
is satisfied for all choices of the parameters $\gamma_1,\gamma_2,c_1,c_2$. Moreover, condition \textbf{(H4)} 
is satisfied if and only if
$\gamma_1\neq -\gamma_2$, while the determinant of the quadratic form $\mathbf S$, at all points of $\mathbb S$, is
$-\frac14 (\gamma_1-\gamma_2)^2$, so that condition \textbf{(H5)} is satisfied as well. 
Indeed, observe that if $\gamma_1\neq\gamma_2$
then the quadratic form $\mathbf S$, at all points of the singular set, is indefinite, while if $\gamma_1=\gamma_2$ 
the vector fields are
equivalent to \eqref{nonhol} and $\mathbf{S}$ vanishes on $\mathbb{S}$. 

Similar computations show also that the nonlinear vector fields
\[
f_1 = \left(
\begin{array}{c}
1\\
0\\ -x_2 + x_3^4
\end{array}
\right) \ , \quad \quad f_2 = \left(
\begin{array}{c}
0\\
1\\ x_1 + x_3^4
\end{array}
\right)
\]
together with the flat constraint $h(x_1,x_2,x_3):=x_3\le 0$ satisfy all assumptions \textbf{(H1)}--\textbf{(H6)}. Indeed, in this case the singular set is
a regular curve,
\[
\mathbb{S}=\{ (x_1,x_2,x_3): x_1=-x_3^4,\ x_2=x_3^4 \}
\]
and 
\[
\nabla h(x)\cdot [f_1,f_2](x)=2\big(1-2x_3^3 (x_1+x_2)\big),
\]
so that on $\mathbb{S}$ one has 
\[
\nabla h(x)\cdot [f_1,f_2](x)=2.
\]
The matrix $\mathbf{S}(x)$ is
\[
\left(
\begin{array}{lr}
4x_3^3(x_3^4-x_2)& 2x_3^3 (x_1-x_2+2x_3^4)\\
2x_3^3 (x_1-x_2+2x_3^4)&4x_3^3(x_3^4+x_1)
\end{array}
\right) 
\]
that vanishes on $\mathbb{S}$. Finally, since 
\[
(\nabla h(x)\cdot f_1(x))^2+(\nabla h(x)\cdot f_2 (x))^2=(x_2-x_3^4)^2+(x_1+x_3^4)^2=\left\| \left(x_1,x_2,x_3\right)-\left(-x_3^4,x_3^4,x_3\right)\right\|^2,
\]
that is larger than or equal to the squared distance from $x$ to $\mathbb{S}$, it is clear that \textbf{(H6)} is satisfied as well.

\subsection{Independence of assumption \textbf{(H6)} from \textbf{(H4)}}
Let
\[
C := \{ x\in\mathbb{R}^3 : h(x)\equiv x_3\le 0\}
\]
together with the dynamics
\[
\dot{x}= f_1(x)u_1 + f_2(x)u_2,\; u = (u_1, u_2) \in \{u\in \mathbb{R}^2: |u|\le 1 \},
\]
where
\[
f_1 = \left(
\begin{array}{c}
1\\
0\\ -x_2
\end{array}
\right) \ , \quad \quad f_2 = \left(
\begin{array}{c}
0\\
1\\ x_1^2
\end{array}
\right).
\]
In this example, $\nabla h(x) = (0, 0, 1)$ and the singular set $\mathbb S$
coincides with the $x_3$-axis.
Hypothesis \textbf{(H4)} is verified, for $\nabla h(x) \cdot [f_1, f_2](x) = 1$ for all $x$.
Recall \textbf{(H6)}: There exist $d_0 > 0$ and $\delta >0$ such that, for all $\bar x\in \mathbb{S}$,
\[
(\nabla h(x) \cdot f_1(x))^2 + (\nabla h(x) \cdot f_2(x))^2  \ge d^2_0 d_{\mathbb{S}}^2(x)\; \text{ for all } x \in \bar x + \delta\mathbb{B}.
\]
In our case, \textbf{(H6)} is not satisfied. Indeed,
\[
(\nabla h(x) \cdot f_1(x))^2 + (\nabla h(x) \cdot f_2(x))^2= x_2^2 + x_1^4,
\]
while
\[
d^2_\mathbb{S}(x) = x_1^2 + x_2^2.
\]

\section{Proof of Theorem \ref{theorem:NFT}}\label{sec:proofNFT}
The proof of Theorem \ref{theorem:NFT} will be divided into some steps, as follows:
\begin{itemize}
\item[1)] Choosing a family of rotational  controls and providing some estimates on the corresponding solutions.
\item[2)] Constructing a trajectory in the interior of the constraint, assuming a second order IPC and choosing rotational controls  acccording with a base drop generator adapted to the quadratic form $\mathbf{S}(x)$ at the points  at the singular set $\mathbb{S}$ or near it.
\item[3)] Concluding the proof of Theorem \ref{theorem:NFT} by combining the above results.
\end{itemize}
\subsection{Implementing rotational controls}\label{subs:screw}
This section is devoted to deriving some estimates along trajectories that correspond to the rotational controls defined in Sect. \ref{sec:rot_contr}.

To state the results, let us choose any drop generator $r$ of period $\tau_0 >0$, and set 
$\mathcal{A}(t):=\mathbf{Area}(R)_0^t$, $R$ being the primitive drop of $r$. 
\begin{proposition}\label{prop:est_trajectory} Assume \textbf{(H1)} and \textbf{(H2)} and choose any drop-generator $r$.
Then there exist a constant $B>0$, independent of $x_0$ in a compact set, and a map $t\mapsto E(t) $ verifying 
$|E(t)|\leq B\|r\|_\infty^2 t^3$ such that, if $k\in \mathbb{Z}$, $\tau\in[0,1]$ 
and we set $|\omega|:=k\tau_0 \tau^{-1}$, $u_\omega(t):=r(\omega t)$, the estimate
				\begin{equation} \label{eq: equality for the feasible trajectory}
					y(\tau) - x_0 =  \tau^2\frac{\mathcal{A}(\tau_0)}{k\tau_0^2}[f_2,f_1](x_0)
					 +  E(\tau).\qquad 
				\end{equation}
holds, where $y(\cdot)$ denotes the solution to the Cauchy problem \eqref{control system}
corresponding to the control $u_\omega$.
\end{proposition}
\begin{proof} For every $t\in[0,\tau]$ we have (writing $u$ in place of $u_\omega$)
\begin{align*}
y(t) - x_0    &=\sum_{\ell =1}^2 {\int_0^t}  f_\ell (y(s)) u^\ell (s) \dd s
			 =
			 \sum_{\ell =1}^2{\int_0^t}  \left[f_\ell (x_0)+(f_\ell (y(s))-f_\ell (x_0))\right] u^\ell (s)  \dd s\\	
			 &=\sum_{\ell =1}^2  \left[f_\ell(x_0)U^\ell(t) +  {\int_0^t}\left( \int_0^s Df_\ell(y(\xi))\dot{y}(\xi) d\xi\right)u^\ell(s)  \dd s\right]\\
			 &=\sum_{\ell =1}^2  \left[f_\ell(x_0)U^\ell(t) +  {\int_0^t}\left( \int_0^s Df_\ell(y(\xi))
\sum_{j=1}^2f_j(y(\xi))u^j(\xi) d\xi\right)u^\ell(s)  \dd s\right]\\
			  &=\sum_{\ell =1}^2  \left[f_\ell(x_0)U^\ell(t) + \sum_{j=1}^2 Df_\ell(x_0) f_j(x_0){\int_0^t} 
			  U^j(s)u^\ell(s)\dd s\right. +\\&
		\left.	\qquad\qquad \qquad  +\sum_{j=1}^2\int_0^t\left( \int_0^s \big(Df_\ell(y(\xi))f_j(y(\xi))- Df_\ell(x_0)f_j(x_0)\big)u^j(\xi) d\xi\right)u^\ell(s)  \dd s\right]\\
	&= \sum_{\ell =1}^2  \left[f_\ell(x_0)U^\ell(t) + \sum_{j=1}^2 Df_\ell(x_0) f_j(x_0)
  	U^{\ell j}(t)\right]+  E(t), \end{align*}
where \[    E(t):= \sum_{j,\ell =1}^2E_{\ell j}(t), \quad E_{\ell j}(t):= 
 \sum_{j=1}^2\int_0^t\left( \int_0^s \big(Df_\ell(y(\xi))f_j(y(\xi))- Df_\ell(x_0)f_j(x_0)\big) 
 u^j(\xi) d\xi\right)u^\ell(s)  \dd s.\]
Observe that for every $\ell,j=1,2$,
\[
| E(t)|\leq \sum_{\ell,j=1}^2 |E_{\ell j}(t)| \leq 4 K \|r\|_\infty^2 \int_0^t \frac{s^2}{2} \dd s = \frac23 K \|r\|_\infty^2 t^3,
\]
where $K$ is a Lipschitz constant\footnote{$K$ exists for $t$ on a compact set,  
because of the Lipschitz continuity of $Df_i$ and $f_i$, $i=1,2$.} for $Df_\ell\cdot f_j (y(\cdot))$, $\ell,j=1,2$.

Hence, by grouping together symmetric and antisymmetric parts,	
\[
\begin{split}
y(t) - x_0 &= \sum_{\ell =1}^2  f_\ell(x_0)U^\ell(t)+ \frac 12 \sum_{j,\ell=1}^2\Big(Df_\ell(x_0) f_j(x_0)
+Df_j(x_0) f_\ell(x_0)\Big)U^{\ell j}_S(t)  \\ 
&\quad + \frac 12 \sum_{j,\ell=1}^2[f_j,f_\ell]U^{\ell j}_A(t) + E(t).
\end{split}
\]	
Therefore, since by $\tau_0$-periodicity and by \eqref{U ell m} it is
\[
 U^j(\tau)=  R^j(k\tau_0) =0\quad\text{and}\quad U^{\ell j}_S(\tau) =\frac12  U^\ell(\tau) U^j(\tau) =0,
\]
recalling also \eqref{area om} one obtains
\[
\begin{split}
y(\tau) - x_0 &= \frac 12 \sum_{j,\ell=1}^2[f_j,f_\ell]\mathbf{Area}(U^\ell,U^j)|_0^\tau +\sum_{j,\ell=1}^2 E^{\ell j}(\tau) \\
&= \frac {\tau^2}{2k^2\tau_0^2} \sum_{j,\ell=1}^2[f_j,f_\ell]\mathbf{Area}(R^\ell,R^j)|_0^{k\tau_0} + \sum_{j,\ell=1}^2E^{\ell j}(\tau)\\
&= \frac {\tau^2}{2k^2\tau_0^2} \sum_{j,\ell=1}^2[f_j,f_\ell]\,k\,\mathbf{Area}({R}^\ell,R^j)|_0^{\tau_0} + \sum_{j,\ell=1}^2E^{\ell j}(\tau)\\
& =   \tau^2 \frac{\mathcal{A}(\tau_0)}{k\tau_0^2}[f_2,f_1] + E(\tau).
\end{split}
\]
The proof is concluded, with $ \displaystyle E(t):= \sum_{j,\ell =1}^2E_{\ell j}(t)$ and 
$B:=\frac23 K$.\end{proof}		

\begin{proposition}\label{proposition: estimate on y in terms of h}
Let us assume hypotheses \textbf{(H1)}--\textbf{(H5)} hold, and let $x_0\in \mathbb{S}$. 
If $\mathbf{S}(x_0)$ is negative semidefinite 
or definite, choose any drop generator $r(.)$ with period $\tau_0$, while, if $\mathbf{S}(x_0)$ is indefinite, choose  
$\displaystyle(\varphi,\beta)\in \left\{\varphi_{\mathbf{S}},-\varphi_{\mathbf{S}} \right\}\times \left\{\beta_{\mathbf{S}},
\frac{\beta_{\mathbf{S}}}{2}\right\}$ and any pointed drop generator $r$  
with period $\tau_0$, phase $\varphi$ and half-amplitude less
than $\beta$. Then there exists $\bar\tau>0$ such that , for all $0<\tau<\bar\tau$,
\begin{equation}\label{stima6} 
h(y_\omega(t)) - h(x_0) \leq  \frac12 \nabla h(x_0) \cdot
	[f_1,f_2](x_0) \frac{1}{\omega^2}\mathbf{Area}(R)_0^{\omega t}\quad \forall t\in[0,\tau ] ,\,\,
\end{equation}
where $y=y_\omega $  is the solution corresponding to the control $u(t) =u_\omega(t) :=r(\omega t)$, with 
$|\omega|=\frac{\tau_0}{\tau}$.
\end{proposition}

\begin{corollary}
Let, in Proposition \ref{proposition: estimate on y in terms of h}, $r(.)$ be the generator of either the drops constructed in Section
\ref{constr_drops}, hence verifying, for some $C,\bar t >0$,
\[
\mathbf{Area}(R)_0^{s}\geq C s^3\quad\forall s\in[0,\bar t].
\]
Set $\hat\tau:=\min\{\bar\tau,\bar t\}$ and let $0<\tau <\hat\tau$ and $|\omega|=\frac{\tau_0}{\tau}$. 
Choose the sign of $\omega$ as 
\begin{equation}\label{sign_om}
\mathrm{sgn}(\omega) = - \mathrm{sgn}\big(\nabla h(x_0) \cdot[f_1,f_2](x_0)\big).
\end{equation}
Then the estimate
\begin{equation}\label{stima7} 
h(y_\omega(t)) - h(x_0) \leq  \frac{C}{2} \nabla h(x_0) \cdot
	[f_1,f_2](x_0)\, \omega t^3\quad \forall t\in[0,\tau ] ,\,\,\forall 0<\tau<\hat\tau
\end{equation}
holds. In particular, recalling assumption \textbf{(H4)}, \eqref{stima7} and our choice of $\omega$ imply
\begin{equation}\label{stima8} 
h(y_\omega(t)) - h(x_0) \leq  
	-\tau_0 C\alpha \frac{t^3}{\tau}\quad \forall t\in[0,\tau ].
\end{equation}
\end{corollary}

\begin{proof}[Proof of Proposition \ref{proposition: estimate on y in terms of h}] 
Since $\nabla h(x_0)\cdot f_\ell(x_0)=0$, $\ell=1,2$, we get, omitting the dependence on $\omega$,
	\begin{align} 
		h(y(t)) - h(x_0) &= \sum_{\ell=1}^{2} \int_{0}^{t} \nabla h(y(s)) \cdot f_\ell(y(s)) u^\ell (s) \dd s  \nonumber\\ 
		&= \sum_{\ell=1}^{2}\nabla h(x_0)\cdot f_\ell(x_0)\int_{0}^{t} u^\ell (s) \dd s 
+\sum_{\ell=1}^{2}\int_{0}^{t}  \Big(\nabla h(y(s)) \cdot f_\ell(y(s))-\nabla h(x_0) 
          \cdot f_\ell(x_0)\Big) u^\ell (s) \dd s  \nonumber\\
			&= \sum_{\ell=1}^{2} \int_{0}^{t} \Big(\nabla h(x_0) \cdot 
                       \left(f_\ell(y(s))- f_\ell(x_0)\right)\Big) u^\ell (s)
                        \dd s\nonumber\\
         &\qquad\qquad\qquad\qquad +\sum_{\ell=1}^{2}\int_{0}^{t}  \Big(\nabla h(y(s)) -\nabla h(x_0)\Big) 
\cdot f_\ell(y(s)) u^\ell (s) \dd s	 \nonumber\\	
			&=
     \sum_{\ell,m=1}^{2}\nabla h(x_0) \cdot \int_{0}^{t}\left( \int_{0}^{s} Df_\ell(y(\xi)) f_m(y(\xi))u^m(\xi)
\dd \xi\right)u^\ell(s)\dd s \nonumber\\
&\qquad\qquad\qquad\qquad+
			\sum_{\ell,m=1}^{2}\int_{0}^{t}\left( \int_{0}^{s}   \frac{\dd}{\dd\xi}\nabla h(y(\xi)) 
\dd\xi\right)f_\ell(y(s))u^\ell(s) \dd s \nonumber\\
&=\sum_{\ell,m=1}^{2}\nabla h(x_0) \cdot
		Df_\ell(x_0)\cdot f_m(x_0) U^{\ell m}(t) \nonumber\\
		&\qquad
		+\sum_{\ell,m=1}^{2}\nabla h(x_0) \cdot \int_{0}^{t}
    \left( \int_{0}^{s} \left(Df_\ell(y(\xi)) f_m(y(\xi)) - Df_\ell(x_0) f_m(x_0))\right)u^m(\xi)\, d\xi\right)u^\ell(s)\dd s \nonumber\\
	&\qquad	
	+	\sum_{\ell,m=1}^{2}\int_{0}^{t}\left( \int_{0}^{s}   D^2h(y(\xi)) f_m(y(\xi))u^m(\xi) d\xi\right)f_\ell(y(s))u^\ell(s) 
      \dd s \nonumber\\
&=\sum_{\ell,m=1}^{2}\nabla h(x_0) \cdot
		Df_\ell(x_0)\cdot f_m(x_0) U^{\ell m}(t) \nonumber\\
		&\qquad
		+\sum_{\ell,m=1}^{2}\nabla h(x_0) \cdot \int_{0}^{t}
    \left( \int_{0}^{s} \left(Df_\ell(y(\xi)) f_m(y(\xi)) - Df_\ell(x_0) f_m(x_0))\right)u^m(\xi)\, d\xi\right)u^\ell(s)\dd s \nonumber\\
	&\qquad + \sum_{\ell,m=1}^{2} D^2h(x_0) f_m(x_0) f_\ell(x_0)U^{lm}(t)\nonumber\\
	&\qquad	
	+	\sum_{\ell,m=1}^{2}\int_{0}^{t}\!\!\!\int_{0}^{s}  \big( D^2h(y(\xi)) f_m(y(\xi)) f_\ell(y(s))
- D^2h(x_0) f_m(x_0) f_\ell(x_0)\big) u^m(\xi)u^\ell(s) 
      \dd\xi\dd s .\nonumber
		\end{align}
Using next the notations introduced in Sect. \ref{sec:rot_contr}, we observe that
\[
\begin{split}
\sum_{\ell,m=1}^{2} 
Df_\ell(x_0)\cdot f_m(x_0) U^{\ell m}(t) &=\frac12  \sum_{\ell,m=1}^{2} [f_m,f_\ell](x_0) U^{\ell m}_A(t)\\
&\quad + \frac12 \sum_{\ell,m=1}^{2}\Big(Df_\ell(x_0)\cdot f_m(x_0)  + Df_m(x_0)\cdot f_\ell(x_0)\Big)U^{\ell m}_S(t).
\end{split}
\]
Therefore, recalling furthermore that $\displaystyle U^{\ell m}_S(t)= \frac12 U^{\ell}(t)U^m(t)$, one gets 
\begin{align}
h(y(t)) - h(x_0) 	&=\frac12 \sum_{\ell,m=1}^{2}\nabla h(x_0) \cdot	[f_m,f_\ell](x_0) U^{\ell m}_A(t)\nonumber\\
&\!\!\!\!\!\!\!\! +\frac14\sum_{\ell,m=1}^{2}\nabla h(x_0) \cdot\left(
		Df_\ell(x_0)\cdot f_m(x_0)+ 	Df_m(x_0)\cdot f_\ell(x_0)\right)U^{\ell}(t)U^m(t)\label{expr_h1}\\
&\!\!\!\!\!\!\!\!  +\frac 12\sum_{\ell,m=1}^{2}  
			D^2h(y(x_0)) f_m(x_0)f_\ell(x_0)
			U^{\ell}(t)U^m(t)\label{expr_h2}\\
& \!\!\!\!\!\!\!\!			+\sum_{\ell,m=1}^{2}\nabla h(x_0) \cdot \int_{0}^{t}
\left( \int_{0}^{s} \left(Df_\ell(y(\xi))\cdot f_m(y(\xi)) - Df_\ell(x_0)\cdot f_m(x_0))\right)u^m(\xi)\, d\xi\right)u^\ell(s)\dd s\nonumber  \\ 
&\!\!\!\!\!\!\!\!	 + 	\sum_{\ell,m=1}^{2}\int_0^t\left(\int_0^s  \Big(D^2h(y(\xi)) f_m(y(\xi))f_\ell(y(s))-D^2h(x_0)
 f_m(x_0)f_\ell(x_0)\Big) u^m(\xi) d\xi\right)u^\ell(s)  \dd s. \nonumber
\end{align}
 Notice that the sum of \eqref{expr_h1} and \eqref{expr_h2} is equal to
$\frac12 \mathbf{S}(x_0)(U(t))$, according to the definitions given in \eqref{equation: S+A italy}
and \eqref{def_Slm}.  Now we are going to use the assumptions that the quadratic form
$\mathbf{S}(x_0)$ is negative semidefinite or definite. In this case,
by the previous estimate we get, independently of the drop generator $r$ that we have chosen to define a rotational
control as in \eqref{eq:rot_controls},
\begin{equation}\label{stima1}	
h(y(t)) - h(x_0) \leq \nabla h(x_0) \cdot
	 [f_1,f_2](x_0) U^{2 1}_{A,\omega}(t) + \sum_{\ell,m=1}^{2}E_{\ell,m}(t),
\end{equation}
where
\begin{align*}
E_{\ell,m}(t)&:=\nabla h(x_0) \cdot \int_{0}^{t}\Big( \int_{0}^{s} \big(Df_\ell(y(\xi)) f_m(y(\xi)) - Df_\ell(x_0)
	 f_m(x_0)\big)u^m(\xi)\, d\xi\Big)u^\ell(s)\dd s  \\
& \qquad  +\int_0^t\left(\int_0^s  \Big(D^2h(y(\xi)) f_m(y(\xi))f_\ell(y(s))-D^2h(x_0) f_m(x_0)f_\ell(x_0)\Big) 
    u^m(\xi) d\xi\right)u^\ell(s)  \dd s. 
\end{align*}
If, instead, the quadratic form  $\mathbf{S}(x_0)$ is indefinite, by using a $\mathbf{S}(x_0)$-adapted  drop $r(t)$
as constructed, e.g., in Sect. \ref{pted_drop}
and in view  of \eqref{S<0}, once again we obtain the estimate \eqref{stima1}.

We turn now to estimating the error term $E_{\ell,m}$. For every $\ell,m=1,2$, let us set 	
\begin{equation}\label{g}	\begin{split}
g_{\ell m}(\xi) &:= \nabla h(x_0)\cdot \big(Df_\ell(y(\xi)) f_m(y(\xi)) -Df_\ell (x_0)  f_m(x_0)\big) \\
&\qquad\quad
+\Big(D^2h(y(\xi)) f_m(y(\xi))f_\ell(y(s))  -D^2h(x_0) f_m(x_0)f_\ell(x_0)\Big).
\end{split}
\end{equation}
If we restrict $g_{\ell m}$ to the compact interval $[0,T]$, recalling that we assumed that  
$h$,$\nabla h$, $D^2h$, $f_1,f_2$, $Df_1$ $Df_2$ are Lipschitz continuous (so that, in particular, 
$y(\cdot)$ is Lipschitz continuous as well), also $g_{\ell m}$ turns out to be $L$-Lipschitz continuous for some $L>0$. 
In particular, since $ g_{\ell m}(0)=0$, we have
\[
|g_{\ell m}(\xi)|\leq L\xi 
\quad\forall \xi\in [0,T].
\] 
Therefore, recalling that $u(t)=r(\omega t)$ $\forall t\in\mathbb{R}$, we get 
\begin{equation}\label{stimaE}\begin{split}
&\!\!\!\!\!\!\!\!\!\!\!\Big|\sum_{\ell,m=1}^{2}	E_{\ell,m}(t)\Big| 
\leq  \sum_{\ell,m=1}^{2} \bigg|\nabla h(x_0) \cdot\!\! \int_{0}^{t}\left( \int_{0}^{s} \left(Df_\ell(y(\xi)) f_m(y(\xi)) 
   - Df_\ell(x_0) f_m(x_0))\right)u^m(\xi)\dd \xi\right)u^\ell(s)\dd s \\
& \qquad   +\int_0^t\left(\int_0^s  \bigg(D^2h(y(\xi)) f_m(y(\xi))f_\ell(y(s))-D^2h(x_0) f_m(x_0)f_\ell(x_0)\bigg) 
   u^m(\xi) \dd\xi\right)u^\ell(s)  \dd s\bigg| \\
& =\sum_{\ell,m=1}^{2}\left|\int_{0}^{t}\left( \int_{0}^{s} g_{\ell m}(\xi) u^m(\xi) \dd\xi\right)u^\ell(s)  \dd s\right|\\
& \leq  \sum_{\ell,m=1}^{2}L\int_{0}^{t}\left( \int_{0}^{s}\xi\, |u^m(\xi)| \dd\xi\right)|u^\ell(s)|  \dd s           \\
&= \frac{L}{|\omega|^3} \sum_{\ell,m=1}^{2}\int_{0}^{|\omega| t}\left( \int_{0}^{|\omega| s}\xi\, |r^m(\xi)| \dd\xi\right)|r^\ell(s)|  \dd s\\
& =  \frac{L}{|\omega|^3}\mathbf{Exc}(R)_0^{|\omega| t}.\end{split}\end{equation}
Now, by \eqref{stima1} and \eqref{stimaE} we obtain
\begin{equation}\label{stima122} 
\begin{split}
 h(y_\omega(t)) - h(x_0) & \leq \nabla h(x_0) \cdot [f_1,f_2](x_0) U^{2 1}_{A,\omega}(t) 
      + \frac{L}{|\omega|^3}\mathbf{Exc}(R)_0^{|\omega| t}\\
&= \nabla h(x_0) \cdot [f_1,f_2](x_0) \frac{1}{\omega^2}\mathbf{Area}(R)_0^{\omega t} 
   + \frac{L}{|\omega|^3}\mathbf{Exc}(R)_0^{|\omega| t}
\end{split}
\end{equation}
Recalling, property (4) in Definition \ref{drop}, there exists a constant $C_r$, depending only on the drop $r$, such that
\[
\mathbf{Exc}(R)_0^{|\omega| t} \le C_r \big|\mathbf{Area}(R)_0^{\omega t}\big|.
\]
Therefore, recalling that from assumption \textbf{(H4)} we have $\big|  \nabla h(x_0) \cdot [f_1,f_2](x_0) \big| \ge \alpha>0$,
by choosing $|\omega|$ large enough (i.e., $\tau>0$ small enough) and recalling that we have set
$\hat\tau =\min\{\bar\tau,\bar t\}$ and we have chosen the sign of $\omega$ as in \eqref{sign_om}, we get 
\[h(y_\omega(t)) - h(x_0) \leq  \frac12 \nabla h(x_0) \cdot
[f_1,f_2](x_0) \frac{1}{\omega^2}\mathbf{Area}(R)_0^{\omega t}\quad \forall t\in[0,\hat \tau/|\omega| ] ,
\]
as soon as  $C_r\leq -\frac{\omega}{2L} \nabla h(x_0) \cdot
[f_1,f_2](x_0)$, in particular, as soon as 
\begin{equation}\label{omegagrande}
|\omega |> \frac{2LC_r}{\alpha}.
\end{equation}
Recalling that $|\omega|=\frac{\tau_0}{\tau}$, the proof is concluded.
\end{proof}
\begin{remark}
	{\rm 
If $\mathbf{S}(x_0)$ is positive (semi)definite,
the sum of \eqref{expr_h1} and \eqref{expr_h2}, that is $\frac12 \mathbf{S}(x_0)(U(t))$, is positive and quadratic 
with respect to $t$. 
Therefore, owing to \eqref{rough_est_area}, there is no way to compensate it through the area term.}
\end{remark}

\subsection{Construction of an interior trajectory}\label{subs:construction}
The present section is devoted to constructing a trajectory $y$ in a suitably small interval $[0,\tau]$ such that $h(y(t))<0$ for all
$t\in (0,\tau]$, with initial point in a suitable neighborhood of the singular set $\mathbb{S}$, 
using the estimates that were obtained in the previous section.

To state the next result, let $x_0\in\mathbb{S}$ and let $\delta >0$ be as given by \textbf{(H6)}. Observe that for each $x$ we can write
\[
\big(\nabla h(x)\cdot f_1(x),\nabla h(x)\cdot f_2(x)\big) = \big|\big(\nabla h(x)\cdot f_1(x),\nabla h(x)\cdot f_2(x)\big)\big|\,(\cos\psi_x,\sin\psi_x),
\] 
with $\psi_x$ locally uniquely and continuously, up to multiples of $2\pi$, determined by the above relation. 
Suppose that at $x_0 \in{\mathbb S}$ the quadratic form
$\mathbf{S}(x_0)$, as defined according to \eqref{def_Slm}, is negative definite or is indefinite.
Then, without loss of generality we may assume that $\mathbf{S}(x)$ is as well
negative definite, resp., indefinite for each $x\in x_0 +\delta\mathbb{B}$. Recalling now that the principal direction 
$\varphi_{\mathbf Q}$ and the half-amplitude $\beta_{\mathbf Q}$ of a quadratic form ${\mathbf Q}$ were defined in \eqref{formaQ},
choose $\beta = \beta (x_0)\in (0, \beta_{{\mathbf S}(x_0)})$ so that, up to possibly taking a smaller $\delta >0$,
\[
\cos\big(\varphi_{\mathbf{S}(x)} - \beta-\psi_{x}\big) \ge \eta_0 >0
\]
for any $x\in x_0+\delta\mathbb{B}$ and a suitable $\eta_0$ independent of $x$, and set 
\begin{equation}\label{def_kx}
k_x:= \begin{cases}1 &\text{in case \textbf{(H5)(i)},} \\ 
\cos\Big(\varphi_{\mathbf{S}(x)} - \beta-\psi_x\Big)&\text{in case \textbf{(H5)(ii)}}.
\end{cases}
\end{equation}
Observe that
\begin{equation}\label{kx>>0}
k_x \ge \eta_0 \quad\text{ for all }x\in x_0+\delta \mathbb{B}.
\end{equation}
Furthermore, for any drop generator $r$ with period $\tau_0$ and any $\omega>1$, let us set 
\begin{equation}\label{def_tauxr}
 \tau_{x,r} :=\min \left\{ \frac{2\eta_0 d_0} {3L}\, ,\,
 \tau_0\right\},
\end{equation}
 where $L$ and $\tau_0$ are as in Definition \ref{drop}, and 
 \begin{equation}\label{def_omxr}
 \bar\omega_{x,r} = \max_{x\in x_0+\varepsilon\mathbb{B}} \frac{2L C}{\big|\nabla h(x) \cdot
 	[f_1,f_2](x)\big|}.
 \end{equation}
\begin{proposition}\label{prop: feasible trajectory}
Let the assumptions \textbf{(H1)}--\textbf{(H6)} hold and let $x_0\in \mathbb{S}$.
For a suitable $\varepsilon>0$ and  for every $x \in x_0 + \varepsilon \B$, there exists  a drop generator $r$ (that possibly
depends on $x$)  
such that, setting $u(t) :=r(\omega t)$ and using $y$ to denote the corresponding solution such that $y(0)=x$, we get
\begin{equation}\label{stima777}h(y(t)) - h(x) \leq  
	 \frac12\nabla h(x) \cdot [f_1,f_2](x) \frac{\mathbf{Area}(R)_0^{\omega t}}{\omega^2},
\end{equation}
for all $|\omega|\geq \bar\omega_{x,r}$ and all $t\in [0,\tau_{x,r}/|\omega|]$, 
provided $\mathrm{sgn}(\omega) = - \mathrm{sgn}\big(\nabla h(x_0) \cdot[f_1,f_2](x_0)\big)$.
\end{proposition}

\begin{remark}
{\rm One can argue from the proof that the statement can be made more precise by saying that the drop 
$r$ can be chosen arbitrarily if hypothesis \textbf{(H5)(i)} %or {\bf (H5)(ii)} are 
is assumed, while, if one posits \textbf{(H5)(ii)}, the choice of 
the phase and half-amplitude of $r$ depends (continuously) on $x$. Actually they depend in a simple way on the phase 
$\varphi_{\mathbf{S}(x)}$ and the half-amplitude $\beta_{\mathbf{S}(s)}$ of the quadratic form $\mathbf{S}(x)$ 
(see points A) and B) of the proof).}
\end{remark}
As in the case of Proposition  \ref{proposition: estimate on y in terms of h}, we can immediately deduce a corollary 
of Proposition  \ref{prop: feasible trajectory} in the case where the primitive drop $R$ is constructed as in 
Section \ref{constr_drops}:
\begin{corollary} 
If  in  Proposition \ref{prop: feasible trajectory} we choose a drop-generator verifying $\mathbf{Area}(R)_0^s\geq Cs^3$ 
$\forall s\in[0,\bar t]$ for some $C,\bar t >0$, and choose the sign of $\omega$ as in \eqref{sign_om},
then \eqref{stima777} becomes  
\begin{equation}\label{stima77} 
h(y(t)) - h(x) \leq  C\frac12 \nabla h(x)\cdot [f_1,f_2](x)\, \omega t^3\quad \forall t\in[0,\tau/|\omega|] ,
\end{equation}
where  $\tau\ge\hat\tau:=\min\{\tau_{x,r},\bar t\}$ and $|\omega|\geq\bar\omega_{x,r}$. 
In particular, setting  $|\omega|= \tau_0 \tau^{-1}$ and keeping the same choice of $\mathrm{sgn}(\omega)$ as above, 
\eqref{stima7} implies
\begin{equation}\label{stima88} 
		h(y_\omega(t)) - h(x) \leq  
			-\alpha(x) \frac{t^3}{\tau}\quad \forall t\in[0,\tau] ,
\end{equation}
with $\alpha(x) := \frac{C\, \tau_0}{2} \big| \nabla h(x) \cdot [f_1,f_2](x)\big|.$
\end{corollary}

\begin{proof}[Proof of Proposition \ref{prop: feasible trajectory}.]
We will consider only the case where $\nabla h(x) \cdot [f_1,f_2](x)<0$. The opposite case can be treated symmetrically.
	
Two scenarios should be considered. The first one is when the initial condition $x\in \mathbb{S}$. 
Then, thanks to our assumptions \textbf{(H1)}--\textbf{(H5)},
the very same construction as in Proposition \ref{proposition: estimate on y in terms of h} can be used.
The second case, that we now examine, is when $x$ belongs to an open neighborhood of $x_0$ but not to $\mathbb{S}$.

Let now the initial point of the trajectory $y(.)$ be $x \in x_0 + \varepsilon \B$ , for some $0< \varepsilon \le 
\delta$, that we will determine later.  We consider the case where 
\begin{equation}\label{stima0}
\displaystyle\Big(\nabla h(x)\cdot f_1(x),\nabla h(x)\cdot f_2(x)\Big)\neq  (0,0)
\end{equation}
i.e., $x\not\in \mathbb{S}$.
Let us select a rotational  control by distinguishing the case \textbf{(H5)(i)} from the case \textbf{(H5)(ii)}.

When $x_0$ satisfies \textbf{(H5)(i)}, let us  choose any  drop generator $r=r_x=(r^1,r^2)$ with a phase 
$\varphi\in[0,2\pi)$ and half-amplitude $\beta\in (0,\pi/2]$ such that
\begin{equation}\label{drop (i)}
\displaystyle\sum_{\ell=1}^{2}\nabla h(x)\cdot f_\ell(x)r^\ell(0) 
         = \min_{|(v_1,v_2)|=1} \sum_{\ell=1}^{2}\nabla h(x)\cdot f_\ell(x)v^\ell.
\end{equation}
Since $r(0) = (\cos(\varphi-\beta),\sin(\varphi-\beta))$ this means that
\[(\cos(\varphi-\beta),\sin(\varphi-\beta)) = 
-\frac{(\nabla h(x)\cdot f_1(x),\nabla h(x)\cdot f_2(x)\big)}{\big|\big(\nabla h(x)\cdot f_1(x),\nabla h(x)\cdot f_2(x)\big)\big|}
\]
Let us observe that if we fix $\beta$ then $\varphi $ is uniquely determined up to multiples of $2\pi$ and, owing to \textbf{(H6)},
\begin{equation}\label{stima32}
\sum_{\ell=1}^{2}\nabla h(x)\cdot f_\ell(x)r^\ell(0) = - \big|\big(\nabla h(x)\cdot f_1(x),\nabla h(x)\cdot f_2(x)\big)\big|\leq -d_0 d_{\mathbb{S}}(x).
\end{equation}
In the case where $x_0$ satisfies \textbf{(H5)(ii)},  (i.e., we can assume that $x$ belongs to a $\delta$-neighborhood of $x_0$ 
where the quadratic form $\mathbf{S}(x)$ is indefinite), 
let us choose a drop generator $r$ with phase $\varphi$ and half-amplitude $\beta$ determined as follows:
\begin{itemize}
\item[A)] if $\Big(\nabla h(x)\cdot f_1(x),\nabla h(x)\cdot  f_2(x)\Big)\cdot  \begin{pmatrix}\cos\left(\varphi_{\mathbf{S}(x)}-\beta_{\mathbf{S}(x)}\right)\\\sin\left(\varphi_{\mathbf{S}(x)}-\beta_{\mathbf{S}(x)}\right)\end{pmatrix}\neq 0$, we set
 \begin{equation}\label{drop (ii) 1}
 (\varphi,\beta):=\left\{\begin{array}{l} \displaystyle\left(\varphi_{\mathbf{S}(x)},\beta_{\mathbf{S}(x)}\right)   \quad\text{if}\quad
	\Big(\nabla h(x)\cdot f_1(x),\nabla h(x)\cdot  f_2(x)\Big)\cdot  \begin{pmatrix}\cos\left(\varphi_{\mathbf{S}(x)}-\beta_{\mathbf{S}(x)}\right)\\\sin\left(\varphi_{\mathbf{S}(x)}-\beta_{\mathbf{S}(x)}\right)\end{pmatrix} < 0\\\\
\displaystyle\left(-\varphi_{\mathbf{S}(x)},\beta_{\mathbf{S}(x)}\right) \quad\text{if}\quad	
\Big(\nabla h(x)\cdot f_1(x),\nabla h(x)\cdot  f_2(x)\Big)\cdot  
\begin{pmatrix}\cos\left(\varphi_{\mathbf{S}(x)}-\beta_{\mathbf{S}(x)}\right)\\\sin\left(\varphi_{\mathbf{S}(x)}-\beta_{\mathbf{S}(x)}\right)\end{pmatrix}  > 0;	   \end{array} \right.
\end{equation}
\item[B)] if, instead,   $\Big(\nabla h(x)\cdot f_1(x),\nabla h(x)\cdot  f_2(x)\Big)\cdot  \begin{pmatrix}\cos\left(\varphi_{\mathbf{S}(x)}-\beta_{\mathbf{S}(x)}\right)\\\sin\left(\varphi_{\mathbf{S}(x)}-\beta_{\mathbf{S}(x)}\right)\end{pmatrix} = 0$, we set
\begin{equation}\label{drop (ii) 2}
(\varphi,\beta):=\left\{\begin{array}{l} \displaystyle\left(\varphi_{\mathbf{S}(x)},\frac12\beta_{\mathbf{S}(x)}\right)   
\quad\text{if}\quad\Big(\nabla h(x)\cdot f_1(x),\nabla h(x)\cdot  f_2(x)\Big)\cdot  
\begin{pmatrix}\cos\left(\varphi_{\mathbf{S}(x)}-\beta_{\mathbf{S}(x)}/2\right)\\\sin\left(\varphi_{\mathbf{S}(x)}-\beta_{\mathbf{S}(x)}/2\right)\end{pmatrix} < 0\\
\displaystyle\left(-\varphi_{\mathbf{S}(x)},\frac12\beta_{\mathbf{S}(x)}\right)   
\quad\text{if}\quad\Big(\nabla h(x)\cdot f_1(x),\nabla h(x)\cdot  f_2(x)\Big)\cdot  
\begin{pmatrix}\cos\left(\varphi_{\mathbf{S}(x)}-\beta_{\mathbf{S}(x)}/2\right)\\\sin\left(\varphi_{\mathbf{S}(x)}-\beta_{\mathbf{S}(x)}/2\right)\end{pmatrix}  > 0.\end{array}\right. 
\end{equation}
\end{itemize}
With these choices, we  obtain, in all cases,
\begin{equation}\label{stima2}
\sum_{\ell=1}^{2}\nabla h(x)\cdot f_\ell(x)r^\ell(0)  =-  k_x\big|\big(\nabla h(x)\cdot f_1(x),\nabla h(x)\cdot f_2(x)\big)\big|
         \leq -k_x\, d_0d_{\mathbb{S}}(x),
\end{equation}
where $k_x$ is as in \eqref{def_kx}.
In particular, taking the control $u(t) := r(\omega t)$ for any $\omega$, using the notation of Section \ref{sec:rot_contr} and 
owing on \textbf{(H6)}, we get
\begin{equation}\label{stima222}\begin{split}
\sum_{\ell=1}^{2}\nabla h(x)\cdot f_\ell(x)U^\ell(t)  
    &= \frac{1}{\omega}\int_0^{\omega t}\sum_{\ell=1}^{2}\nabla h(x) \cdot f_\ell(x)r^\ell(\xi)\dd\xi\\
   & = \frac{1}{\omega}\int_0^{\omega t}\sum_{\ell=1}^{2}\nabla h(x)\cdot f_\ell(x)\left(r^\ell(0) 
           +\int_0^\xi \dot{r}^\ell (s)\dd s \right)\dd\xi \\
&\leq \sum_{\ell=1}^{2}\nabla h(x) \cdot f_\ell(x) r^\ell(0)\, t +  \frac{L\omega }{2}\big|(\nabla h(x)\cdot f_1(x),\nabla h(x)\cdot f_2(x))\big|\, t^2\\
&\leq d_{\mathbb{S}}(x) \, t\, \Big(- k_x d_0+ \frac{L }{2}\omega t\Big).
\end{split}
\end{equation}
Using the same rotational controls, we also obtain
	\begin{align*} 
		&\!\!\!\!\!\!\!\!\!\!\!\!\!\! h(y(t)) - h(x) = \sum_{\ell=1}^{2}\int_{0}^{t}  \nabla h(y(s)) \cdot f_\ell(y(s)) u^\ell(s) \dd s  \\ 
		&= \sum_{\ell=1}^{2}\nabla h(x)\cdot f_\ell(x)U^\ell(t)  +\sum_{\ell=1}^{2}\int_{0}^{t}  \Big(\nabla h(y(s)) \cdot f_\ell(y(s))-\nabla h(x) \cdot f_\ell(x)\Big) u^\ell(s) \dd s  \\ 	
		&=\sum_{\ell=1}^{2}\nabla h(x)\cdot f_\ell(x)U^\ell(t)  + \sum_{\ell,m=1}^{2}\nabla h(x) \cdot \int_{0}^{t}\left( \int_{0}^{s} Df_\ell(y(\xi)) f_m(y(\xi))u^m(\xi)\, d\xi\right)u^\ell(s)\dd s \\ 
		&\qquad\qquad\qquad	+
		\sum_{\ell=1}^{2}\int_{0}^{t}\left( \int_{0}^{s}   \frac{\dd}{\dd\xi} \nabla h(y(\xi)) \dd\xi\right)f_\ell(y(s))u^\ell(s) \dd s\\
			&= \sum_{\ell=1}^{2}\nabla h(x)\cdot f_\ell(x)U^\ell(t)  +\sum_{\ell,m=1}^{2}\nabla h(x) \cdot
			Df_\ell(x) f_m(x) U^{\ell m}(t) \\
			&\qquad
			+\sum_{\ell,m=1}^{2}\nabla h(x) \cdot \int_{0}^{t}\left( \int_{0}^{s} \big(Df_\ell(y(\xi)) f_m(y(\xi)) - Df_\ell(x) f_m(x)\big)u^m(\xi)\, d\xi\right)u^\ell(s)\dd s \\
			&\qquad  
			+	\sum_{\ell,m=1}^{2}\int_{0}^{t}\left( \int_{0}^{s}   D^2h(y(\xi)) f_m(y(\xi))u^m(\xi) d\xi\right)f_\ell(y(s))u^\ell(s) \dd s .
		\end{align*}
Since  	
\[
\displaystyle\sum_{\ell,m=1}^{2} 
		Df_\ell(x)f_m(x) U^{\ell m}(t) = \sum_{\ell,m=1}^{2} \bigg[\frac12 [f_m,f_\ell](x) U^{\ell m}_A(t) + \frac12\Big(Df_\ell(x) f_m(x) + Df_m(x)f_\ell(x)\Big)U^{\ell m}_S(t)\bigg]
\]
and $ \displaystyle U^{\ell m}_S(t)= \frac12 U_{\ell}(t)U_{\ell }(t)$, as we computed in Section \ref{sec:rot_contr}, 
and recalling the definition of the quadratic form ${\mathbf S}$ that was given in \eqref{def_Slm}, one gets, for all 
$\displaystyle
|\omega| \geq \frac{2LC}{\big|\nabla h(x_0) \cdot [f_1,f_2](x_0)\big|}$ and $t\leq \tau_0/|\omega|$, with
$\mathrm{sgn}(\omega)$ as in \eqref{sign_om},
\begin{equation}\label{stima1222}		
\begin{split}
&\!\!\!\!\!\!\!\! h(y(t)) - h(x) = \sum_{\ell,m=1}^{2}\nabla h(x) \cdot
				\frac12 [f_m,f_\ell](x) U^{\ell m}_A(t)+\sum_{\ell=1}^{2}\nabla h(x)\cdot f_\ell(x)U^\ell(t) + \frac 12 \mathbf{S}(x)\big(U(t),U(t)\big)\\
			&\;	\qquad	+\sum_{\ell,m=1}^{2}\int_0^t\!\!\int_0^s \nabla h(x) \cdot  \Big(Df_\ell(y(\xi)) f_m(y(\xi)) - Df_\ell(x) f_m(x)\Big) u^m(\xi) u^\ell(s)\dd\xi\dd s 	\\
&\;\qquad  + 	\sum_{\ell,m=1}^{2}\int_0^t\!\!\int_0^s  \Big(D^2h(y(\xi)) f_m(y(\xi))f_\ell(y(s))-D^2h(x) f_m(x)f_\ell(x)\Big) u^m(\xi) u^\ell(s)
      \dd\xi   \dd s\\
			&\quad\leq  \nabla h(x) \cdot
				[f_1,f_2](x) U^{2 1}_A(t)+\sum_{\ell=1}^{2}\nabla h(x) f_\ell(x)U^\ell(t)	+  \mathbf{S}(x)\big(U(t),U(t)\big) ,
\end{split}
\end{equation}
where the last inequality has been obtained arguing as in \eqref{stima1}-\eqref{omegagrande}, with $x$ in place of $x_0$.
In view of estimates \eqref{stima222} and \eqref{stima1222}, 
we then obtain, for all $\displaystyle	\omega \geq  \bar\omega_{r,x}$ and $t\leq \tau_0/\omega$,
\[	
h(y(t)) - h(x) \leq 
		 \nabla h(x)\cdot
		[f_1,f_2](x) U^{2 1}_A(t)+ d_{\mathbb{S}}(x) \, t\, \Big(-\eta_0 d_0+ \frac{L }{2}\omega t\Big)+\mathbf{S}(x)\big(U(t),U(t)\big) .
\]
In case $x_0$ verifies \textbf{(H5)(ii)}, by our choice of the drop generator $r$, we have 
$\mathbf{S}(x)\big(U(t),U(t)\big)\leq 0$ for all $t\in\mathbb{R}$.
Therefore we get, for every $\displaystyle
|\omega| \geq  \bar\omega_{r,x}$ and $t\leq \min\big\{\frac{\tau_0}{|\omega|}, \frac{2\eta_0 d_0} {L\omega}\big\}$,
\[
h(y(t)) - h(x) \leq  \nabla h(x) \cdot
		[f_1,f_2](x) U^{2 1}_A(t) = \nabla h(x) \cdot [f_1,f_2](x) \frac{\mathbf{Area}(R)_0^{\omega t}}{\omega^2},
\]
i.e. the thesis  is proved.

In the (more delicate) case where $x_0$ verifies \textbf{(H5)(i)}, let $y_0\in\mathbb{S}$ be such that $d_{\mathbb S}(x)=|x-y_0|$.  There are two possibilities: 1) $y_0$ satisfies
\textbf{(H5)(i)}, and then we choose a drop $r$ as in \eqref{drop (i)} with $y_0$ in place of $x$; or, 2), $y_0$ satisfies \textbf{(H5)(ii)},
and we choose a drop satisfying \eqref{drop (ii) 1} or \eqref{drop (ii) 2} with $y_0$ in place of $x$, taking 
into account if either A) or B) are in place. 
In both cases $\mathbf{S}(y_0)\big(U(t),U(t)\big)\leq 0$ for all  $t\in\mathbb{R}$. Thus, recalling \eqref{stima222},
\begin{equation}
\begin{split}\label{est_H5i}
\sum_{\ell=1}^{2}\nabla h(x) f_\ell(x)U^l(t)	+  \mathbf{S}(x)\big(U(t),U(t)\big) 
         &\leq  d_{\mathbb{S}}(x) \, t\, \Big(- k_x d_0+ \frac{L }{2}\omega t\Big)
           + \left|\big(\mathbf{S}(x)-\mathbf{S}(y_0)\big)\big(U(t),U(t)\big)\right|  \\
& \leq  d_{\mathbb{S}}(x) \, t\, \Big(- k_x d_0+ \frac{L }{2}\omega t\Big)  + L d_{\mathbb{S}}(x) \omega t^2\\
&\le d_{\mathbb{S}}(x) \, t\, \Big(- \eta_0 d_0+ \frac32 L \omega t\Big) ,
\end{split}
\end{equation}
where the last inequality follows from \eqref{kx>>0}.
Hence, for all $\displaystyle
|\omega| \geq \bar\omega_{x,r}$ and $t\leq \min\big\{\frac{\tau_0}{|\omega|}, \frac{2\eta_0 d_0} {3L|\omega|}\big\}$,
\[
h(y(t)) - h(x) \leq  
\frac12\nabla h(x) \cdot
[f_1,f_2](x) U^{1 2}_A(t) = \frac12\nabla h(x) \cdot
[f_1,f_2](x) \frac{\mathbf{Area}(R)_0^{\omega t}}{\omega^2},
\]
namely the thesis is proved in this case as well.
\end{proof}
\begin{remark}
{\rm If $\mathbf{S}(x_0)$ is positive semidefinite, in particular vanishes (as it happens, e.g., 
for the nonholonomic integrator
with a parabolic constraint), it is clear from \eqref{est_H5i} that the assumption \textbf{(H6)} is crucial in order 
to compensate
the error term $\left|\big(\mathbf{S}(x)-\mathbf{S}(y_0)\big)\big(U(t),U(t)\big)\right|$. On the contrary, if 
$\mathbf{S}(x_0)$ is indefinite or negative definite at all $x_0\in \mathbb{S}$, then an inspection into \eqref{stima222}
shows that the assumption \textbf{(H6)} may be weakened (the linearity w.r.t. $d_\mathbb{S}$ is no longer necessary).}
\end{remark}
\subsection{Proof of Theorem \ref{theorem:NFT}}\label{subs:completion} 
Let $(\bar{x},\bar{u})$ be a reference process with initial condition ${x}_0$ such that $h({x}_0) \le 0$.
Recall that the time $\tau_1$ was defined at \eqref{eq: expression of the violating time}.
\begin{proof}[Proof of Theorem \ref{theorem:NFT}.]
We will construct first the trajectory $y$ satisfying \eqref{estimate: on first interval} and \eqref{estimate: on second interval} provided the final time $T$ is small enough (independently of the initial condition ${x}_0$). Next we will prove
estimates \eqref{eq:nonlinest} and \eqref{W11est} and, finally, we will remove the additional assumption on $T$.

\medskip

By compactness (of a suitable neighborhood of the reference trajectory $\bar{x}$) the numbers $\varepsilon$, constructed in
Proposition \ref{prop: feasible trajectory}, and $\bar{\tau}_{x,r}$, resp. $\bar{\omega}_{x,r}$, defined at \eqref{def_tauxr},
resp. \eqref{def_omxr}, are independent of the initial point $x_0$ that was considered there.
Let
\[
\bar{x}_1:=\bar{x}(\tau_1).
\]
Two cases may occur:
\begin{itemize}
\item[(a)] $d_{\mathbb{S}}(\bar{x}_1)\ge \varepsilon$;
\item[(b)] $d_{\mathbb{S}}(\bar{x}_1)< \varepsilon$.
\end{itemize}

\textit{Case} (a). In this case the first order IPC is satisfied in a neighborhood of $\bar{x}_1$, and it is well known
that the trajectory $y$ with the required properties can be constructed (see, e.g., \cite{Rampazzo_Vinter_1999} or \cite{Rampazzo_Vinter_2000}). Indeed, $T$ can be chosen to be small enough, uniformly with respect to the initial condition $x_0$
that belongs to a compact set, so that the trajectories of interest remain bounded away from $\mathbb{S}$.

\smallskip

\textit{Case} (b). On the time interval $[0, T ]$, we consider the control $(w_1 , w_2)$ defined as 
\begin{align}\label{eq:choice_controls}&   w_i(t) := \bar u_i(t) \chi_{[0,\tau_1)}(t) + 
u^{\omega}_i(t)  \chi_{[\tau_1,\tau_1+\tau(d)]}(t) + \bar u_i\big(t-\tau(d)\big)\chi_{(\tau_1+\tau(d),T]}(t),\quad i=1,2
\end{align}
where $u^{\omega}_i(\cdot )$ are the controls as defined in (\ref{eq:rot_controls}), where $r$, $\omega$, and $\tau(d) \ge 0$ are
to be specified later  ---in particular,  
$\tau(\cdot)$ is  continuous and verifies $\tau(0)=0$ and $\tau(d) >0$ for $d>0$.

Let $y(\cdot)$ be the trajectory corresponding to the control $(w_1 , w_2)$.
Owing to the definition of $\tau_1$, inequality \eqref{estimate: on first interval} is automatically satisfied on $[0,\tau_1]$,
since $y(t)= \bar x(t)$ in that interval. Set now $\bar{x}_1:=\bar{x}(\tau_1)$ and take the drop generator $r$ and 
$\bar\tau:=\tau_{r,\bar{x}_1}$ as in Proposition \ref{prop: feasible trajectory}.
Then, taking also $\omega$ as in Proposition \ref{prop: feasible trajectory},
inequality \eqref{stima88} shows that $h(y(t)) < 0$ on $(\tau_1,\tau_1+\bar\tau]$. 
Moreover, by \eqref{eq: equality for the feasible trajectory} the trajectory $y(.)$ in $[\tau_1,\tau_1+\bar\tau]$, with $y(\tau_1)=\bar{x}_1$, 
verifies
\begin{equation}\label{eq: expression of y at the boundary point}
 		y(\tau_1+\bar\tau) = \bar x_1 +  \bar\tau^2\frac{\mathcal{A}(\tau_0)}{k\tau_0^2}[f_2,f_1](\bar x_1)
					 + E(\bar \tau), 
\end{equation}
where we recall that, by Proposition \ref{prop:est_trajectory}, $|E(\bar\tau)|\le B \|r\|^2_\infty \bar\tau^3$ and the constant $B$ is independent
of $x_0$ in a compact set.

We deal now with some estimates on $y$ in the interval $(\tau_1+\bar\tau, T]$.
For simplicity of notation, we shall write \[f(x,w) := f_1(x)w_1+f_2(x)w_2.\]
For all $t\ge \tau_1$, 
the Lipschitz continuity of $f_1, f_2$ 
(with common Lipschitz constant $k_f$) and (\ref{eq: expression of y at the boundary point}) imply,
for a suitable constant $\alpha_0>0$ independent of $x_0$, $\bar\tau$ small enough, and $\omega $,
\begin{align*}
 		\big\|y(t+\bar\tau) - \bar x(t)\big\| & \leq \big\|y(\tau_1+\bar\tau) - \bar x_1\big\|  + \int_{\tau_1}^{t} \big\|f(y(s+\bar\tau), w(s+\bar\tau) )- f(\bar x(s), \bar u(s))\big\| \, \dd s \\
		 & = \big\|y(\tau_1+\bar\tau) - \bar x_1\big\|  +\int_{\tau_1}^{t} \big\|f(y(s+\bar\tau), \bar u(s)) - f(\bar x(s), \bar u(s))\big\| \, \dd s \\ 
		 & \leq \alpha_0\,  \bar\tau^2 + k_f \int_{\tau_1}^{t} \big\|y(s+\bar\tau) - \bar x (s)\big\| \, \dd s.
\end{align*}
 By Gronwall's Lemma, we therefore obtain
 \begin{equation}\label{eq: estimate on y and x bar}
 		\big\|y(t+\bar\tau) - \bar x(t)\big\| \le \alpha_0\, \bar\tau^2 e^{k_f(t-\tau_1)} \qquad \text{for all } t\ge \tau_1.
 \end{equation}
Consequently, recalling \eqref{eq:choice_controls} with $\bar\tau$ in place of $\tau(d)$,
\begin{equation}\label{est_y'}
\begin{split}
\int_{\tau_1}^T \big\| \dot{y}(t+\bar\tau)-\dot{\bar{x}}(t)\big\|\, \dd t&=\int_{\tau_1}^T \big\| f(y(t+\bar\tau),\bar u(t))-f(\bar x(t),\bar u(t))\big\|\, \dd t\\
&\le k_f \int_{\tau_1}^T \big\| y(t+\bar\tau) - \bar x(t)\big\|\, \dd t\\
&\le k_f \int_{\tau_1}^T\alpha_0\, \bar\tau^2 e^{k_f(t-\tau_1)}\,\dd t\\
&=  \alpha_0\, \bar\tau^2\,  \big( e^{k_f (T-\tau_1)}-1\big).
\end{split}
\end{equation}
By evaluating $h(y(t))$ for $t\ge \tau_1+\bar\tau$, we obtain
\begin{align*}
 		h(y(t)) &  = h(y(\tau_1 + \bar\tau)) + \int_{\tau_1 + \bar\tau}^{t} \nabla h(y(s)) \cdot f(y(s) , \bar u(s-\bar\tau))  \dd s \\ 
		& =  h(y(\tau_1 + \bar\tau)) + \int_{\tau_1}^{t-\bar\tau}  \nabla h(y(s+\bar\tau)) \cdot f(y(s+\bar\tau) , \bar u(s))  \dd s
\shortintertext{(by Taylor development of the integral term around $\bar{x}(s)$)}
&=h(y(\tau_1 + \bar\tau)) + \int_{\tau_1}^{t-\bar\tau} \nabla h(\bar x(s)) \cdot [f_1(\bar x(s)) \bar u_1(s) + f_2(\bar x(s)) \bar u_2(s)   ]  \dd s\\
&\qquad +   \int_{\tau_1}^{t-\bar\tau} \big(y(s+\bar\tau)-\bar x(s)\big)\cdot \\
&\qquad\qquad \cdot \bigg(\sum_{i=1}^{2} \int_{0}^{1} \Big[D^2 h \big(\bar x(s) + r (y(s+\bar\tau)-\bar x(s))\big)  f_i\big(\bar x(s) + 
                r (y(s+\bar\tau)-\bar x(s))\big) \\ 
&\qquad\qquad \;\;\, + \nabla h \big(\bar x(s) + r (y(s+\bar\tau)-\bar x(s))\big) Df_i \big(\bar x(s) + 
         r (y(s+\bar\tau)-\bar x(s))\big) \, \dd r\Big] \bigg)\, \bar u_i(s) \dd s\\
& \le h(y(\tau_1 + \bar\tau)) +  h(\bar x(t-\bar\tau)) - h(\bar x_1)  +   \int_{\tau_1}^{t-\bar\tau} K \big| y(s+\bar\tau)-\bar x(s)\big|  \dd s
\end{align*}
for a suitable constant $K$ that can be computed using \textbf{(H1)}--\textbf{(H3)}. Since $h(\bar x(t-\bar\tau)) \le d$ and $h(\bar x_1)=0$, recalling
(\ref{eq: estimate on y and x bar}), we obtain
\begin{equation} \label{eq: estimate on y in terms of h proof of NFT} 
h(y(t)) \le h(y(\tau_1+\bar\tau)) + d + Tb_0\bar\tau^2 \qquad \text{ for all } t \ge \tau_1 + \bar\tau,
\end{equation}
for a suitable constant $b_0>0$ depending only on $h,f_1,f_2$.
Therefore, from (\ref{stima88}), evaluated at $\tau_1+\bar\tau$ with the initial time $0$ replaced by
$\tau_1$, we deduce that
\[ 
h(y(t)) \leq -\gamma\alpha\,\bar\tau^2  + d + Tb_0\bar\tau^2 .
\]
Choosing now 
\begin{equation}\label{ch_T}
T\le\frac{\gamma\alpha}{4b_0}
\end{equation}
and 
\begin{equation}\label{ch_tau}
\bar\tau=\tau (d):= \min\left\{\sqrt{\frac{2d}{\gamma\alpha}},1\right\},
\end{equation}
we obtain that (\ref{estimate: on second interval}) is verified also on $(\tau_1+\tau(d), T]$.
 		 		
To conclude the proof, we shall derive the estimates  on the $L^\infty$ and the $W^{1,1}$-distance between 
$y(.)$ and $\bar x(.)$, in terms of the extent $d$ of violation of the state constraint, as well as \eqref{contrest}.

Since $y=\bar x$ on $[0,\tau_1]$, we obtain
\[ \| y-\bar x\|_{L^\infty(0,T)} =  \|  y-  {\bar x}\|_{L^\infty(\tau_1,\tau_1+\tau(d))} + \|  y- {\bar x}\|_{L^\infty(\tau_1+\tau(d),T)} \ . \]
On $(\tau_1, \tau_1+\tau(d)]$, using the bound $k_0$ on the vector fields and $\bar x(\tau_1)=y(\tau_1)$, we obtain by Gronwall lemma
\begin{align*}  \|  y-  {\bar x}\|_{L^\infty(\tau_1,\tau_1+\tau(d))}  \leq 2k_0 \tau(d).  \end{align*}
On $(\tau_1+\tau(d),T]$, we deduce from \eqref{eq: estimate on y and x bar} and \textbf{(H1)} that
\[
\begin{split}
   \|  y(t)-  {\bar x}(t)\|&\le  \| y(t)-\bar x(t-\tau(d))\| + \| \bar x(t-\tau(d))-\bar x(t)\|\\
   &\leq \alpha_0\, \tau(d)^2  e^{k_f(T-\tau_1)} + k_0 \tau(d). 
\end{split}   
\]
We conclude, owing to the choice of $\tau(d)$ in (\ref{ch_tau}), that for some constant $K>0$ depending only on $T, f_1,f_2$ we have
\[ \| y(.)-\bar x(.)\|_{L^\infty(0,T)}  \leq K\sqrt{d}, \]
provided $d\le 1$.

We deal now with the $L^1$ estimate on the derivatives. Recalling \textbf{(H1)} and \eqref{est_y'}, one has
\[
\begin{split}
\big\| \dot{y}(.)-\dot{\bar x}(.)\big\|_{L^1(0,T)}&=\big\|\dot{y}(.)-\dot{\bar x}(.)\big\|_{L^1(\tau_1,\tau_1+\tau(d))}
        +\big\|\dot{y}(.)-\dot{\bar x}(.-\tau(d))\big\|_{L^1(\tau_1+\tau(d),T)}\\
        &\qquad \qquad +\big\|\dot{\bar x}(.-\tau(d))-\dot{\bar x}(.)\big\|_{L^1(\tau_1+\tau(d),T)}\\
&\le 2k_0 \tau(d) +   \alpha_0\, \tau^2(d)\,  \big( e^{k_f (T-\tau_1)}-1\big)+\big\|\dot{\bar x}(.-\tau(d))-\dot{\bar x}(.)\big\|_{L^1(\tau_1+\tau(d),T)}\\
&= o(1)\;\text{ as $d\to 0$.}
\end{split}
\]
This establishes \eqref{W11est}, while \eqref{contrest} follows immediately from \eqref{eq:choice_controls}.
Finally, the statement on the $W^{1,1}$-norm of $y-\bar x$ in case $\dot{\bar x}$ has bounded variation 
can be deduced easily from \eqref{W11est},
by representing each component of $\dot{\bar x}$ as a difference of nondecreasing functions and computing explicitly 
$\|\dot{\bar x}(.)-\dot{\bar x}(.-\tau(d))\|_{L^1(\tau_1+\tau(d),T)}$.
		
The extra assumption \eqref{ch_T} on the length of the time interval can be removed by repeating the above construction 
finitely many times, as in \cite[Proof of Theorem 2.1]{Rampazzo_Vinter_1999}.
\end{proof}
\begin{remark} {\rm 
A construction similar to, e.g., \cite{Rampazzo_Vinter_2000}, that provides a general $W^{1,1}$-estimate of the distance from
$y$ and $\bar x$ under a first order IPC, cannot be applied here. Indeed, the time delay of the reference control
$\bar u$, that we use in \eqref{eq:choice_controls} following \cite{Soner_1986,motta,Rampazzo_Vinter_1999}, 
is essential in order to obtain \eqref{estimate: on second interval}.
Actually, deleting the time delay introduces in the estimate of $h$ along the trajectory $y$ an error that is proportional to 
$\tau(d)$, i.e., to $\sqrt{d}$, and not to $\tau^2(d)$. 
This error cannot be compensated by the gain towards the interior of the constraint that is
obtained using the rotational strategy \eqref{eq:rot_controls}. This gain is computed at \eqref{stima88},
and, at the time $\tau(d)$, turns out to be proportional to $\tau^2(d)$.}
\end{remark}

\section{An application to PDE} \label{sec:applications}
As an application of Theorem \ref{theorem:NFT}, we will prove the following results, which deal with the uniqueness 
of the  (constrained viscosity) solution to a  Hamilton-Jacobi-Bellman equation verified by  the value of an
optimal control problem with state constraints.  
The classical, first order  results in this direction were proved in H.~M. Soner's seminal papers
\cite{Soner_1986} \cite{Soner_1986_II}. We shall limit ourselves to an infinite horizon problem but, with trivial adjustments, one can prove  akin results  for the finite horizon problem.

Let $\Ell : \mathbb{R}^n\times U\to \mathbb{R}$ be a function satisfying:
\begin{itemize}
\item[\textbf{(H$_\Ell$)}] $\Ell$ is bounded and continuous and is uniformly Lipschitz continuous with respect to $x$,
with Lipschitz constant $k_\Ell>0$. 
\end{itemize}
For all $\xi\in C$, define
\[
U_\xi := \{ u:[0,+\infty)\to U \, |\, x_\xi^u(t)\in C \forall t \in [0,+\infty) \}.
\]
Observe that in our setting $U_\xi$ is nonempty, by well known viability results. Define also
\[
V(\xi):= \inf\Big\{ \int_0^{+\infty} e^{-t} \Ell\big(x_\xi^u(t),u(t)\big)\dd t: \ u(.)\in U_\xi\ \Big\}.
\]
Then
\begin{theorem}\label{theorem:cont_val}
Assume \textbf{(H1)}--\textbf{(H6)} and \textbf{(H$_\Ell$)}, and, furthermore, that $\partial C$ is compact. 
Then $V(.)$ is bounded and uniformly continuous on $C$. 
\end{theorem}
As an immediate consequence of Corollary IV.5.9 in \cite{BCD} and of Theorem \ref{theorem:cont_val}, one obtains
\begin{theorem}\label{th:uniqHJ}
Under the assumptions of Theorem \ref{theorem:cont_val}, the value function $V$ is the unique bounded and 
uniformly continuous
\textrm{constrained viscosity solution} of the Hamilton--Jacobi equation
\begin{equation}\label{HJ_eq}
V+H(x,DV)=0
\end{equation}
in the whole of $C$, where $H(x,p)=	\sup_{u=(u_1,u_2)\in U}\big\{-\big(f_1(x)u_1+f_2(x)u_2\big)\cdot p+\Ell (x,u)\big\}$.
\end{theorem}
For the meaning of \textit{constrained viscosity solution} on $C$ for \eqref{HJ_eq}, we make reference to 
\cite[Definition IV.5.6]{BCD}\footnote{ Let us just only mention to the crucial fact that while  the notion of {\it viscosity solution} is given through the simultaneous validity of the supersolution and subsolution relations at each point of the domain, in the case of the {\it constrained viscosity solution}  at the points in the domain's boundary the function is only required to be a supersolutiom.}.

\begin{proof}[Proof of Theorem \ref{theorem:cont_val}.] 
Since the integrand $\mathscr{\Ell}$ is bounded, say by a constant $A>0$, the value function is obviously bounded 
(by $A$ itself). 

Fix now $r>0$ and let $x_1,x_2\in C$ be such that $|x_1-x_2|<r$. Let furthermore $\delta >0$, 
together with $t^\ast$ such that $A e^{-t}<\frac{\delta}{2}$. Then there exists a control $u\in U_{x_1}$ such that
\begin{equation}\label{est_val}
\int_0^{t^\ast} e^{-t}\Ell \big(x_{x_1}^u (t),u(t)\big) \dd t + A e^{-t^\ast} < V(x_1)+ \delta.
\end{equation}
Consider the trajectory $x_{x_2}^u(.)$ and observe that, for all $t\in [0,t^\ast]$ one has
\[
h(x_{x_2}^u(t))\le h(x_{x_1}^u(t))+k_h |x_{x_2}^u(t))-x_{x_1}^u(t))\le |x_2-x_1| \big(e^{k_f t^\ast} -1\big).
\]
By Theorem \ref{theorem:NFT}, there exist a control $u^\ast\in U_{x_2}$ and a constant $K$ such that
\[
\big\| x_{x_2}^{u^\ast}(.)-x_{x_2}^u(.)\big\|_{L^\infty(0,t^\ast)}\le K \sqrt{r}.
\]
Then, by the Dynamic Programming Principle (see, e.g., Proposition IV.5.5 in \cite{BCD})
\begin{align*}
V(x_2)&\le \int_0^{t^\ast} e^{-t} \Ell \big(  x_{x_2}^{u^\ast}(t),u^\ast(t)\big)\dd t + 
                           e^{-t^\ast} V\big( x_{x_2}^{u^\ast}(t^\ast)\big)\\
&\le \int_0^{t^\ast} e^{-t} \Ell \big(  x_{x_2}^{u^\ast}(t),u^\ast(t)\big)\dd t + A e^{-t^\ast}\\
&\le \int_0^{t^\ast} e^{-t} \Ell \big(  x_{x_2}^{u}(t),u^\ast(t)\big)\dd t + 
k_\Ell K(1-e^{-t^\ast})  \sqrt{r}\\
&\qquad   +\int_0^{t^\ast}\!\! e^{-t^\ast}\Big(\Ell \big(  x_{x_2}^{u}(t),u^\ast(t)\big)-
\Ell \big(  x_{x_2}^{u}(t),u(t)\big)\Big)\dd t +A e^{-t^\ast}\\
&\le  \int_0^{t^\ast} e^{-t} \Ell \big(  x_{x_1}^{u}(t),u(t)\big)\dd t + A e^{-t^\ast}+k_\Ell  (1-e^{-t^\ast})r+
k_\Ell K(1-e^{-t^\ast})  \sqrt{r}\\
&\quad +\!\! \int_0^{t^\ast}\!\! e^{-t}\Big(\Ell \big(  x_{x_1}^{u}(t),u^\ast(t)\big)-
\Ell \big(  x_{x_1}^{u}(t),u(t)\big)\Big)\dd t
+\!\!\int_0^{t^\ast}\!\! e^{-t}\Big(\Ell \big(  x_{x_2}^{u}(t),u^\ast(t)\big)-
\Ell \big(  x_{x_2}^{u}(t),u(t)\big)\Big)\dd t \\
&\le V(x_1) + \delta + o(1)
\end{align*}
as $r\to 0$, where in the last inequality we have used \eqref{est_val} and \eqref{contrest}.
Since $\delta$ is arbitrary and $x_1$ and $x_2$ can be interchanged, this shows that $V$ is continuous
in the whole of $C$. Now, owing to Theorems 3.1 and 4.1 in \cite{motta}, that hold under the assumption that 
$\partial C$ is compact, the value function is indeed uniformly continuous in the interior of $C$.
The proof is concluded.
\end{proof}
\section{Appendix: Proof of Proposition \ref{inv_S}}
\begin{proof}
Let $x$ be a chart on some open set $O$. For any smooth  $1$-form $\omega = \sum_{j=1}^n \omega_j \dd x^j$  and any smooth vector field $F=\sum_{j=1}^n F^j \frac{\partial}{\partial x^j}$, the quantity 	$\Phi:O\to\R$, $\Phi(x):=\omega(x)\cdot F(x)$ is invariant under the smooth local change of coordinates $\tilde{x}=\tilde{x}(x)$.
Namely, if for every $j=1,\dots,n$,  $\tilde{\omega}_j(\tilde{x})=\sum_{r=1}^n \frac{\partial x^r}{\partial \tilde{x}^j}\omega_r(x(\tilde{x}))$
and $\tilde{F}^j(\tilde{x})=\sum_{\ell=1}^n\frac{\partial \tilde{x}^j}{\partial x^\ell} F^\ell (x(\tilde{x}))$, one has 
\[
\tilde\Phi(\tilde x(x)) := \tilde\omega(\tilde x(x))\cdot \tilde F(\tilde x(x)) = \omega(x)\cdot F(x)
=\Phi(x).
\]
Let now $h:\R^n\to\R$ be of class ${\mathcal C}^1$ and let  the vector fields $f_i:\R^n\to\R^n$, $i=1,2$
		 be of class ${\mathcal C}^2$. 
		Fix $w=(w^1,w^2)\in\R^2$ and consider the function $\Phi_w :O\to\R$ defined  as
		\[
	\,\,\,\,\, \Phi_w (y) := \sum_{i=1}^2 w^i \big( \nabla h (y) \cdot f_i(y)\big) \quad\forall y\in O.
		\]
		In view of the above consideration, for each smooth curve $x:[0,T]\to O$, the function
		\[
		\Phi_w\circ x :[0,T]\to\R,
		\]
		is independent of the choice of coordinates in $\R^n$. In particular, the scalar function
$\frac{\dd }{\dd t} \big( \Phi_w\circ x(t)\big)_{|t=0}$ is independent of coordinates. 
Let now $u_1, u_2:[0,T]\to\R^2$ be continuous and fix an initial point $x_0\in O$. For $\tau\in ]0,T]$ 
sufficiently small consider the unique local solution $x:[0,\tau]\to\R^n$ of the Cauchy problem
\[\dot{x}=\displaystyle{\sum_{i=1}^2 f_i (x)u^i}\,\,\,\quad
			x(0)=x_0
\]
and set $(w^1,w^2):= ( u^1(0),u^2(0))$. Then
\begin{align*}
\frac{\dd }{\dd t}^+ \big( \Phi_w\circ x(t)\big)_{|t=0}&=\sum_{j=1}^2 u^j(0)\, 
\nabla \big( \nabla h (x(t))\cdot f_j(x(t))\Big)_{|t=0}\,\Big( \sum_{i=1}^2 f_i(x(t)) u^i(t)\Big)_{|t=0}\\
			&= \sum_{i,j=1}^2 \left( \sum_{k = 1}^n \frac{\partial}{\partial x^k} 
\left( \sum_{\ell = 1}^n\frac{\partial h}{\partial x^\ell} f^\ell_j\right)(x_0)f_i^k(x_0)\right) w^i w^j\\
			&= \sum_{i,j=1}^2\left(\sum_{k,\ell=1}^n
\left(\frac{\partial^2 h}{\partial x^k\partial x^\ell}f_j^\ell f_i^k
+
\frac{\partial h}{\partial x^\ell}\frac{\partial f_j^\ell}{\partial x^k} f_i^k\right)\right)(x_0)\,  w^i w^j\\
		&= \sum_{i,j=1}^2\big(f_j D^2 h f_i + \nabla h \cdot D f_j f_i\big) (x_0) \, w^i w^j\\
			&= \sum_{i,j=1}^2\left(f_j D^2 h f_i + \frac12 \nabla h \cdot \big( D f_j f_i+ D f_i f_j\big)\right) (x_0) \, w^i w^j\\
			&= \sum_{i,j=1}^2 S_{ij}(x_0) w^i w^j.
\end{align*}
Similarly (and with an obvious meaning of the notation) one obtains
	\[	\frac{\dd }{\dd t}^+ \Big( \tilde\Phi_w\circ \tilde x(t)\big)_{|t=0}= \sum_{i,j=1}^2 \tilde S_{ij}(x_0) w^i w^j.\]
	Therefore 
	\[
	 \sum_{i,j=1}^2 S_{ij}(x_0) w^i w^j=	\frac{\dd }{\dd t}^+ \Big( \Phi_w\circ x(t)\big)_{|t=0} = 	\frac{\dd }{\dd t}^+ \Big( \tilde\Phi_w\circ \tilde x(t)\big)_{|t=0}= \sum_{i,j=1}^2 \tilde S_{ij}(\tilde x(x_0)) w^i w^j.
	\]
		Since  this holds for any pair $(w^1,w^2)\in\R^2$ (and the matrices   $S,\tilde S$ are symmetric)  this implies $\tilde S_{ij}(x_0)= S_{ij}(\tilde x(x_0))$ for all $i,j=1,2$ (and all $x_0\in O$.)
	\end{proof}
\noindent\textit{Acknowledgment.} Nathalie T. Khalil started working on this problem as part of her PhD thesis, 
during a visit to the department of mathematics at the University of Padova, that was funded by 
Universit\'e de Bretagne Occidentale and by Universit\'e de Bretagne Loire, France.


\begin{thebibliography}{99}
		
	\bibitem{Arutyunov_Aseev_1997} Arutyunov, A. V., \& Aseev, S. M. (1997). Investigation of the degeneracy phenomenon of the maximum principle for optimal control problems with state constraints. SIAM J. Control and Optimization, 35(3), 930-952.
	
	\bibitem{Arutyunov_Karamzin_2020_nondegeneracy} Arutyunov, A. V., \& Karamzin, D. Y. (2020). A survey on regularity
conditions for state-constrained optimal control problems and the non-degenerate maximum principle, J. Optim. Theory Appl. 184(3), 697-723.

\bibitem{BCD} Bardi, M., Capuzzo-Dolcetta, I. (1997). Optimal Control and Viscosity Solutions of Hamilton-Jacobi-Bellman Equations.
Birkh\"auser.
		
	\bibitem{Bettiol_Bressan_Vinter_2010} Bettiol, P., Bressan, A., \& Vinter, R. (2010). On trajectories satisfying a state constraint: $W^{1,1}$ estimates and counterexamples. SIAM J. Control and Optimization, 48(7), 4664-4679.
	
	\bibitem{Bettiol_Facci_2014} Bettiol, P., \& Facchi, G. (2014). Linear estimates for trajectories of state-constrained differential inclusions and normality conditions in optimal control. J. Math. Anal. Appl., 414(2), 914-933.
	
	\bibitem{Bettiol-Vinter2010-sensitivity analysis} Bettiol, P., \& Vinter, R. B. (2010). Sensitivity interpretations of the costate variable for optimal control problems with state constraints. SIAM journal on control and optimization, 48(5), 3297-3317.
	
	\bibitem{Bettiol_Khalil_2017}  Bettiol, P. \& Khalil, N. (2017). Non-degenerate forms of the generalized Euler-Lagrange condition for state-constrained optimal control problems: In Imaging and Geometric Control. Variational Methods. 18 vols. Radon Series on Computational and Applied Mathematics. De Gruyter, 297-315.
	
	\bibitem{Bettiol_Khalil_Vinter_2016} Bettiol, P., Khalil, N., \& Vinter, R. B. (2016). Normality of generalized Euler-Lagrange conditions for state constrained optimal control problems. J. Convex Anal, 23(1), 291-311.
	
	\bibitem{Bressan_Facci_2011} Bressan, A., \& Facchi, G. (2011). Trajectories of differential inclusions with state constraints. J. Differential Equations, 250(4), 2267-2281.
	
	\bibitem{Fereira_Fontes_Vinter_1999}  Ferreira, M.M.A., Fontes, F.A.C.C., \& Vinter, R. (1999). Nondegenerate necessary conditions for nonconvex optimal control problems with state constraints, J. Math. Anal. Appl. 233(1), 116-129.
	
	\bibitem{Fereira_Vinter_1994} Ferreira, M. M. A., \& Vinter, R. B. (1994). When is the maximum principle for state constrained problems nondegenerate?, J. Math. Anal. Appl., 187(2), 438-467.
	
	\bibitem{Frankowska_2010_survey} Frankowska, H. (2010). Optimal control under state constraints. Proceedings of the International
Congress of Mathematicians. Vol. IV, 2915-2942, Hindustan Book Agency, New Delhi.
	
	\bibitem{Frankowska_Vinter_2000} Frankowska, H., \& Vinter, R. B. (2000). Existence of neighboring feasible trajectories: applications to dynamic programming for state-constrained optimal control problems. J. Optim. Theory Appl., 104(1), 20-40.
	
	\bibitem{Khalil_2017_thesis} Khalil, N. (2017). Optimality conditions for optimal control problems and applications (Doctoral dissertation, 
Universit\'e\ de Bretagne occidentale-Brest).
	
	\bibitem{Lopes_Fontes_dePinho_2011} Lopes, S. O., Fontes, F. A. C. C., \& de Pinho, M. D. R. (2011). On constraint qualifications for nondegenerate necessary conditions of optimality applied to optimal control problems. Discr. Cont. Dynamical Systems-A, 29(2), 559-575.
	
	\bibitem{motta} Motta, M. (1995). On nonlinear optimal control problems with state constraints. SIAM J. Control Optim., 33(5), 1411-1424.
	
\bibitem{petr} Petrov, N. N. (1968). Controllability of autonomous systems, Differ. Uravn. (Differential Equations), 4(4), 606-617.

	\bibitem{Rampazzo_Vinter_1999} Rampazzo, F., \& Vinter, R. B. (1999). A theorem on existence of neighbouring trajectories satisfying a state constraint, with applications to optimal control. IMA Journal of Mathematical Control and Information, 16(4), 335-351.
	
	\bibitem{Rampazzo_Vinter_2000} Rampazzo, F., \& Vinter, R. (2000). Degenerate optimal control problems with state constraints. SIAM J. Control Optim., 39(4), 989-1007.
	
	\bibitem{Soner_1986} Soner, H. M. (1986). Optimal control with state-space constraint I. SIAM J. Control Optim., 24(3), 552-561.
	
	\bibitem{Soner_1986_II} Soner, H. M. (1986). Optimal control with state-space constraint. II. SIAM J. Control Optim., 24(6), 1110-1122.
	
	\bibitem{Vinter_book_2010} Vinter, R. (2010). Optimal control. Springer Science \& Business Media.
	
\end{thebibliography}
\end{document}